\documentclass{article}
\usepackage{amsmath,amsthm,amssymb}
\usepackage{graphicx}

\newtheorem{theorem}{Theorem}[section]
\newtheorem{lemma}[theorem]{Lemma} 
\newtheorem{propose}[theorem]{Proposition}

\newtheorem{example}[theorem]{Example} 
\newtheorem{remark}[theorem]{Remark} 

\numberwithin{equation}{section}

\newcommand{\refpart}[1]{{\it (#1)}}  

\renewcommand\ge\geqslant
\renewcommand\geq\geqslant
\renewcommand\le\leqslant
\renewcommand\leq\leqslant

\newcommand{\CC}{\mathbb{C}}
\newcommand{\PP}{\mathbb{P}}
\newcommand{\RR}{\mathbb{R}}
\newcommand{\QQ}{\mathbb{Q}}
\newcommand{\ZZ}{\mathbb{Z}}

\newcommand{\EE}{{\mathcal E}}
\newcommand{\SU}{{\mathcal S}}

\newcommand{\frc}[2]{{\textstyle \frac{#1}{#2}\, }}

\newcommand{\hpg}[5]{{}_{#1}\mbox{\rm F}_{\!#2}\!
	\left(\left.{#3 \atop #4}\right| #5 \right) }

\newcommand{\hpgo}[2]{{}_{#1}\mbox{\rm F}_{\!#2}}

\title{\bf Belyi maps from zeroes of hypergeometric polynomials}

\author{Raimundas Vidunas\\
\em Institute of Applied Mathematics, Vilnius University}

\date{\empty}

\begin{document}

\maketitle


\begin{abstract}
Evaluation of low degree hypergeometric polynomials 
to zero defines an algebraic hypersurface in the affine space of the free parameters and the argument. 
This article investigates the algebraic surfaces $\hpgo21(-N,b;c;z)=0$ for $N=3$ and $N=4$.
As a captivating 
application, these surfaces parametrize certain families of genus 0 Belyi maps.
\end{abstract}

\section{Introduction}
\label{sec:intro}

Systematic cataloging of Belyi maps is an active research undertaking \cite{BelyiDB,BelyiAS22,vHV15}. 
Contributing to this venture, 
this article considers comprehensively Belyi maps $\varphi:\PP^1\to\PP^1$ of the form
\begin{equation} \label{eq:g11hm}
\varphi(x)=(1-x)^p(1-\lambda x)^q\,G_m(x)^r
\end{equation}
or 
\begin{equation} \label{eq:g2hm}
\varphi(x)=(1+\alpha x+\beta x^2)^p\,G_m(x)^r,
\end{equation}
where 
$G_m(x)$ is a polynomial of degree $m$ without multiple roots, and 
\begin{equation} \label{eq:bmpws}
\varphi(x)=1+O(x^{m+2}).
\end{equation}
This implies $G_m(0)=1$. 
The powers $p,q,r$ are allowed to be positive or negative integers.
If the powers 
are positive integers, then the considered Belyi maps 
are called {\em Shabat polynomials} \cite[\S 2.2]{LandoZvonkin} or {\em generalized Chebyshev polynomials} \cite{Adrd4}.
The powers $p,q,r$ are assumed to be different.

Assuming that the point $x=\infty$ is above $\varphi\in \{0,\infty\}$,
the distinct points in the three canonical fibers $\varphi\in\{0,1,\infty\}$ are:
\begin{itemize}
\item two roots of $(1-x)(1-\lambda x)$ or $1+\alpha x+\beta x^2$, of branching order $|p|$ or $|q|$,
above $\varphi=0$ or $\varphi=\infty$ depending on the signs of $p,q$;
\item the roots of $G_m(x)$, of branching order $|r|$, above $\varphi\in \{0,\infty\}$;
\item the point $x=\infty$, of branching order $|p+q+mr|\neq0$ or $|2p+mr|\neq 0$ respectively for the two forms, 
above $\varphi\in \{0,\infty\}$;
\item the point $x=0$, of branching order $m+2$, above $\varphi=1$;
\item 
$d-(m+2)$ non-branching points above $\varphi=1$,
where $d$ is the degree of $\varphi$.
\end{itemize}
In total we count 
\begin{equation} \label{eq:tp}
2+m+1+1+(d-m-2) = d+2
\end{equation}
distinct points in the 3 fibers, which is the minimal number of points in 3 fibers
for a covering $\PP^1\to\PP^1$ by the Riemann-Hurwitz formula \cite[Theorem 5.9]{Silverman}, 
and the number of points that Belyi maps of genus 0 must have in the 3 fibers.
This will be recapped in Lemma \ref{th:np}. 

As it turns out, the considered families of Belyi maps are parametrized 
by the zero sets of hypergeometric polynomial equations,
namely
\begin{equation} \label{eq:g11hm21}
\hpg21{-m-1,\,\frac{q}{r}}{-m-\frac{p}{r}}{\,\lambda}=0
\end{equation}
for the Belyi maps of the form (\ref{eq:g11hm}), and
\begin{equation}  \label{eq:g2hm21}
\hpg21{-\frac{m+1}2,-\frac{m}2}{-m-\frac{p}{r}}{\frac{4\beta}{\alpha^2}}=0
\end{equation}
for the form (\ref{eq:g2hm}). 
These hypergeometric equations were derived earlier by Adrianov \cite[Propositions 3.5, 3.8]{Adrd4} 
in the context of Shabat polynomials.
The equations give the generic number of distinct Belyi maps (up to M\"obius transformations)
of the considered branching patterns, namely $m+1$ for fixed $p,r,q,m$ in the form (\ref{eq:g11hm}),
or $\lceil m/2 \rceil$ for fixed $p,q,m$ in the form (\ref{eq:g2hm}). As we show in Lemmas \ref{th:ck11} and \ref{th:ck2}, 
this number is smaller exactly when one of the quotients 
$p/r$, $q/r$ is a smaller positive integer.
It appears that the hypergeometric polynomials in (\ref{eq:g11hm21}), (\ref{eq:g2hm21}) are irreducible over $\QQ$ typically,
giving the maximal Galois orbits of Belyi maps with the considered branching patterns.

The considered Belyi maps constitute a borderline easy case
that continues extensively the examples  in \cite[\S 2.2]{LandoZvonkin} and \cite{APZ}.
The new examples and methodology will not be radically novel for active readers of \cite{APZ}, \cite{LandoZvonkin}.
In particular, we encounter new families of Belyi maps defined over $\QQ$ that are parametrized by:
\begin{itemize}
\item rational points on an elliptic curve, as in \cite[\S 2.2.4.3]{LandoZvonkin};
see Examples \ref{eq:pqrr113}, \ref{ex:ec5} and \ref{ex:ec6} here.
\item integer solutions of Pell's equation, as in \cite[Ch.~10]{APZ}; 
see Example \ref{ex:pell}.
\item rational points on a cubic surface or a cubic pencil; see Examples \ref{ex:pq2}, \ref{ex:surf4}.
\item integral points on elliptic curves or elliptic surfaces; see Example \ref{ex:degd4}.
\item conic curves defined over $\QQ$ that may have no rational points; see Example \ref{ex:conic}.
\end{itemize}

The investigation of this article starts with recalling the relevant properties of hypergeometric 
functions. Sections \ref{sec:hpg3} and \ref{sec:hpg4} considers the algebraic surfaces defined
\begin{equation}
\hpg21{-N,\,b\,}{c}{z}=0,
\end{equation}
with $N=3$ and $N=4$.  We modify the lower parameter $-c+1-N$, 
so that after clearing the denominators we have the symmetric identity
\begin{equation} \label{eq:hpgsym1}
(c)_m\;\hpg21{-N,\,b}{-c+1-N}{z}=(b)_m\,z^N\;\hpg21{-N,\,c}{-b+1-N}{\frac1z\,}
\end{equation}
of the same hypergeometric summation in the opposite directions. 
The parameters $b,c$ are thereby interchangeable in combination with the M\"obius transformation $z\mapsto 1/z$.
These equations with $N\in\ZZ$ define algebraic surfaces in the affine space 
defined by the coordinates $b,c,z$. 
Further symmetries of these hypergeometric polynomials are presented in (\ref{eq:hpgsym2})--(\ref{eq:hpgsym6}).

The technical contribution of this article is careful consideration of the hypergeometric polynomial equations,
including their degenerations, and analysis of their algebraic structure and elliptic fibrations.

The article continues immediately with a primer on hypergeometric 
functions and their properties that will used. 
In the next section we start with several easier examples of Belyi maps that are largely covered in \cite[\S 2.2]{LandoZvonkin}.
Thereby we gradually build up the basic methodology, illustrate arising technical details,
and set the context of where this article goes beyond the established understanding.

\section{Hypergeometric details}
\label{sec:append}

The Gauss hypergeometric function \cite[Ch.~2]{AAR}  
is defined by the series
\begin{align} \label{eq:hpg}
\hpg{2}{1}{a,\,b\,}{c}{\,z} = \sum_{k=0}^{\infty}\frac{(a)_k\,(b)_k}{(c)_k\,k!}\,z^k.
\end{align}
Here $(\alpha)_k = \alpha\,(\alpha+1)(\alpha+2)\ldots(\alpha+k-1)$
denotes the {\em Pochhammer symbol},  or the ``raising" variant of the factorial. 
Using the $\Gamma$-function \cite[Ch.~1]{AAR} 
we can write
\begin{align}
(\alpha)_k = \, \frac{\Gamma(\alpha+k)}{\Gamma(\alpha)}. 
\end{align}
The standard analytic continuation of the $\hpgo21$-function is onto $\CC\setminus[1,\infty)$. 
This  
function is undefined when $c$ equals zero or a negative integer. 

\subsection{Hypergeometric polynomials}

If well-defined, the hypergeometric series (\ref{eq:hpg}) is a polynomial 
in $z$ when $a$ or $b$ are non-positive integers.  
In addition to (\ref{eq:hpgsym1}), we have these symmetries of hypergeometric polynomials:
\begin{align}  \label{eq:hpgsym2}
(c)_N\;\hpg21{-N,\,b}{-c+1-N}{z} 
& = (-1)^N(b+c)_N\;\hpg21{-N,\,b\,}{b+c}{1-z} \\
& = (b+c)_N\,(-z)^N\;\hpg21{-N,\,c\,}{b+c}{1-\frac1z\,} \\
& = (c)_N\,(1-z)^N\;\hpg21{-N,\,e}{-c+1-N}{\frac{z}{z-1}}  \\
 \label{eq:hpgsym6}
& = (b)_N\,(z-1)^N\;\hpg21{-N,\,e}{-b+1-N}{\frac1{1-z}}\!.
\end{align}
where $e=-b-c+1-N$.
The following expression defines a polynomial of degree $\lfloor N/2 \rfloor$ 
for any integer $N\ge 0$:
\begin{equation}
\hpg21{-\frac{N}2,-\frac{N-1}2}{c}{z}
\end{equation}

If $c$ is a non-negative integer, 
but $a$ or $b$ is a non-negative integer $\ge c$, 
the hypergeometric summation (\ref{eq:hpg}) 
can be considered as a polynomial, of degree $-a$ or $-b$.
This interpretation is often used 
in theory of orthogonal polynomials \cite{KS10}. 
We adopt this interpretation of hypergeometric functions like 
\begin{equation} \label{eq:degenpo}
\hpg21{-n,\,b}{-n-m}{z}
\end{equation}
with integer $n,m\ge 0$.
The well-known identities 
\begin{align} \label{frlin2f1}
\hpg{2}{1}{a,\,b\,}{c}{\,z} =&\, (1-z)^{-a}\;\hpg{2}{1}{a,\,c-b\,}{c}{\,\frac{z}{z-1}} \\
\label{eq:eulr}
= &\, (1-z)^{c-a-b}\;\hpg{2}{1}{c-a,\,c-b\,}{c}{\,z}
\end{align}
of Pfaff and Euler \cite[Th. 2.2.5]{AAR}  
should not be automatically applied then; 
see \cite[Lemma 3.1]{degen}.
Yet these identities are valid when all three numbers $a,b,c$ are non-negative integers,
and $-c\ge \max(-a,-b)$; see \cite[\S 9]{degen}.

\subsection{Distinctive roots of hypergeometric polynomials}

Distinctiveness of roots of hypergeometric polynomials is easily proved by 
using {\em contiguous relations}  \cite[\S 2.5, \S3.7]{AAR} of $\hpgo21$-functions.
In particular, we use the observations in the following two lemmas.
\begin{lemma}  \label{th:ck11}
Consider the sequence of hypergeometric polynomials
\begin{equation} \label{eq:defpk}
P(k)=\hpg21{-k,\,b}{1-k-c}{z},
\end{equation}
with some $b,c\in\CC$. 
Suppose that for a positive integer $k$ we have $(b)_k\neq 0$, $(c)_{k}\neq 0$,  and $(b+c)_k\neq 0$.
Then:
\begin{enumerate}
\item The polynomials $P(k)$ and $P(k-1)$ do not have common roots.
\item The polynomial $P(k)$ has $k$ distinct roots. 
\item $z=1$ is not a root of $P(k)$.
\end{enumerate}
\end{lemma}
\begin{proof}
The sequence $P(k)$ satisfies the recurrence relation
\begin{equation} \label{eq:recpk}
\frac{k(k-1+b+c)z}{k-1+c}\,P(k-1) = (kz+k+bz+c)\,P(k) - (k+c)\,P(k+1).\ 
\end{equation}
This follows from the contiguous relations \cite{contgr03}; 
or by applying Zeilberger's algorithm  \cite[\S 3.11]{AAR}; or by checking the series expansion.
We apply the recurrence with decreasing $k$, and use the assumption that $(b+c)_{k}\neq 0$.
The recurrence implies that a common root of $P(k)$ and $P(k-1)$ is either $z=0$, 
which conflicts with the hypergeometric values $1$ at $z=0$,
or it would be a common root of $P(k-2),\ldots$, and of $P(0)=1$ contradictorily. 
 
The second statement follows from the degree of $k$ of $P(k)$ thanks to \mbox{$(b)_{k}\neq 0$,}
and the fact that another contiguous relation
\begin{align}
\frac{1-z}{k}\,\frac{dP(k)}{dz}+P(k) = \frac{k-1+b+c}{k-1+c}\,P(k-1)
\end{align}
implies that a multiple root of $P(k)$ would be also a root of $P(k-1)$, leading to the established first statement.

The last statement follows from the Chu-Vandermonde identity \cite[p.~67]{AAR}
\begin{equation} \label{eq:chvdm}
\hpg21{-k,B}{C}{1}=\frac{(C-B)_k}{(C)_k},
\end{equation}
giving the value $(b+c)_k/(c)_k$ for $P(k)$ evaluated at $z=1$.
\end{proof}

\begin{lemma}  \label{th:ck2}
Consider the sequence of hypergeometric polynomials
\begin{equation} 
P(k)=\hpg21{-\frac{k}2,\,-\frac{k-1}2}{1-k-c}{z},
\end{equation}
with some $c\in\CC$. 
Suppose that for a positive integer $k$ we have $(c)_{k}\neq 0$ and $(2c)_k\neq 0$.
Then:
\begin{enumerate}
\item The polynomials $P(k)$ and $P(k-1)$ do not have common roots.
\item The polynomial $P(k)$ has $\lfloor k/2 \rfloor$ distinct roots.
\item $z=1$ is not a root of $P(k)$.
\end{enumerate}
\end{lemma}
\begin{proof}
The sequence $P(k)$ satisfies the recurrence relation
\begin{equation}
P(k)-P(k+1) = \frac{kz\,(2c+k-1)}{4(c+k)(c+k-1)} \, P(k-1).
\end{equation}
Under the assumption that $(b+c)_{k}\neq 0$, the recurrence shows that 
a common root of $P(k)$ and $P(k+1)$ is either the rejectable $z=0$, 
or it would be a common root of $P(k-1)$, $P(k-2),\ldots$, and of $P(0)=1$ contradictorily. 
 
The second statement follows from the degree of $k$ of $c_k$, thanks to \mbox{$(a)_{k}\neq 0$,}
and the fact that another contiguous relation
\begin{align}
\frac{2(1-z)}{k}\frac{dP(k)}{dz} + P(k) = \frac{2c+k-1}{2(k+c-1)}\,P(k-1)
\end{align}
implies that a multiple root of $P(k)$ would be also a root of $P(k+1)$, 
leading to the established first statement.

The last statement follows from applying the Chu-Vandermonde identity  (\ref{eq:chvdm}). 
The evaluation at $z=1$ is
\begin{equation}
\left. \left( {\textstyle c+\frac12}\,\right)_{\lfloor \frac{k}2 \rfloor} \,
\right/  \left(c+\left\lceil {\textstyle \frac{k}2 } \right\rceil\right)_{\lfloor \frac{k}2 \rfloor}
\end{equation}
for both even and odd $k$.
\end{proof}


\section{Cubic hypergeometric polynomials}
\label{sec:hpg3}

After clearing the denominators in the equation $\;\hpg21{-3,\;b}{-c-2}{z}=0\;$ we obtain
\begin{equation} \label{eq:hpg3}
b(b+1)(b+2)z^3+3bc(b+1)z^2+3bc(c+1)z+c(c+1)(c+2)=0,
\end{equation}
or rather more compactly,
\begin{equation}  \label{eq:hpg3a}
(bz+c)^3+3(bz+c)(bz^2+c)+2(bz^3+c)=0.
\end{equation}
Let $\SU_3$ denote the algebraic surface defined by (\ref{eq:hpg3}), of degree 6.

Fixing $z$ leads to a cubic equation in $b,c$. 
For $z\in\{0,1\}$, the equation factorizes into 3 linear factors.  
For $z\in\{-1,2,\frac12\}$, it factorizes into a quadratic and a linear parts. 
For other fixed $z$, we get an irreducible cubic curve of genus $0$.
Here is a parametrization of those cubic curves by $e=bz+c$:
\begin{equation}  \label{eq:bcst}
b=\frac{e\,(e+1)\,(e+2)}{z\,(1-z)\,(3e+2z+2)}, \qquad c=\frac{e\,(e+z)\,(e+2z)}{(z-1)\,(3e+2z+2)}.
\end{equation}
This also gives a birational parametrisation of 
$\SU_3$ by $z$ and $e$. If \mbox{$3e+2z+2=0$} in the denominators,
then the polynomial in (\ref{eq:hpg3}) can be factored as 
\begin{equation} \label{eq:efa} \textstyle
-\frac{4}{27}(z+1)(z-2)(2z-1).
\end{equation}
Both numerators in (\ref{eq:bcst}) equal 0 as well at the exceptional 3 points $(e,z)\in\{(0,-1),(-2,2),(-1,\frac12)\}$.
The 3 points are blown up to these lines on $\SU_3$:
\begin{equation} \textstyle
b=c,z=-1;  \qquad 2b+c+2=0,z=2; \qquad b+2c+2=0,z=\frac12.
\end{equation}

We will be interested in the points on $\SU_3$ with both $b,c$ positive, while \mbox{$z\not\in\{0,1\}$} and $b\neq c$. 
Dependence of the signs of $b,c$ on the parameters $e,z$ by the parametrization (\ref{eq:bcst}) 
is depicted in Figure \ref{fig:sregions}\refpart{i}.
There are two regions marked by + where $b$ and $c$ are positive; they are delineated by 
\begin{equation} \label{eq:shb1}
z<0 \qquad \mbox{and} \qquad e\,(3e+2z+2)<0.
\end{equation}
The exceptional lines (\ref{eq:efa}) do not apply to this interest.

\begin{figure}
\begin{center}
\begin{picture}(400,175)
\put(-3,2){\includegraphics[width=175pt]{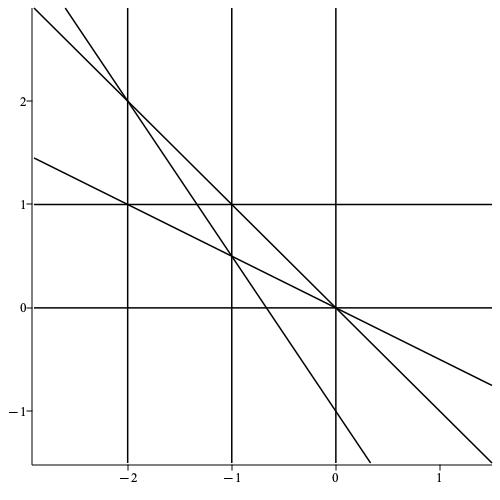}}
\put(223,2){\includegraphics[width=175pt]{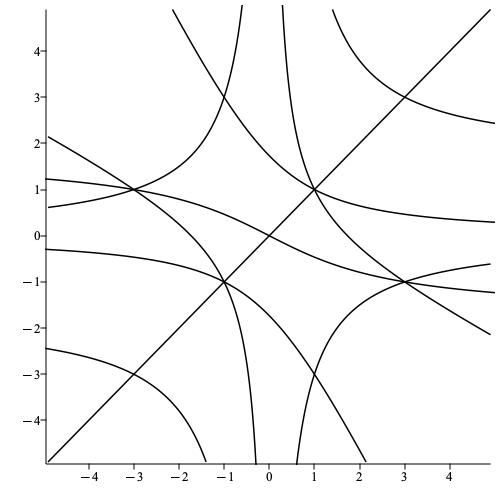}}
\put(165,6){\small $e$}  \put(0,170){\small $z$} 
\put(391,6){\small $y$}  \put(232,170){\small $t$} 
\put(23,41){\small $=$} \put(56,40){\small $\pm$} \put(90,36){\small $=$} \put(104,54){\small $+$} 
\put(115,17){\small $+$} \put(126,35){\small $=$}  \put(158,33){\small $\mp$} \put(151,58){\small $=$}  
\put(23,84){\small $\pm$} \put(54,84){\small $=$} \put(79.5,73){\tiny $\pm$} \put(90,72){\small $\mp$} 
\put(59,98){\small $\mp$} \put(72,99){\tiny $\pm$} \put(83,87){\small $=$} \put(102,91){\small $\mp$}
\put(15,108){\small $\mp$}  \put(23,131){\small $=$}  \put(46,112){\small $\pm$} \put(66,108){\tiny $\mp$}
\put(15,170){\tiny $\mp$} \put(30,165){\small $\pm$}  \put(58,145){\small $=$} \put(93,145){\small $\pm$}
 \put(139,87){\small $\pm$} \put(139,145){\small $\mp$}  %
\put(295,160){\small $=$}  \put(311,143){\small $\pm$}   \put(335,134){\small $\mp$} 
\put(361,160){\small $=$} \put(382,143){\small $\mp$} \put(361,112){\small $=$} 
\put(330,93){\small $\mp$}  \put(361,88){\small $\pm$}  \put(388,77){\small $=$}  
\put(324,70){\small $=$} \put(361,44){\small $\pm$}   \put(388,66){\small $\mp$}   
\put(317,38){\small $\pm$} \put(334,22){\small $=$} 
\put(292,47){\small $\mp$} \put(265,25){\small $=$} \put(246,37){\small $\mp$} 
\put(265,65){\small $=$}  \put(298,91){\small $\mp$}  \put(265,95){\small $\pm$}  
\put(241,107){\small $=$}  \put(241,117){\small $\mp$}   
\put(303,109){\small $=$} \put(265,137){\small $\pm$}   
\put(-2,-5){\refpart{i}} \put(230,-5){\refpart{ii}} 
\end{picture}
\end{center}
\caption{The signs of $b$ and $c$ depending: \refpart{i} on $e,z$ in formula (\ref{eq:bcst});
\refpart{ii} on $t,y$ in formula (\ref{eq:bcuv}). The upper sign in $\pm,\mp,=$ is for $b$,
while the lower one is for $c$. There sign + indicates both $b,c$ positive.}
\label{fig:sregions}
\end{figure}
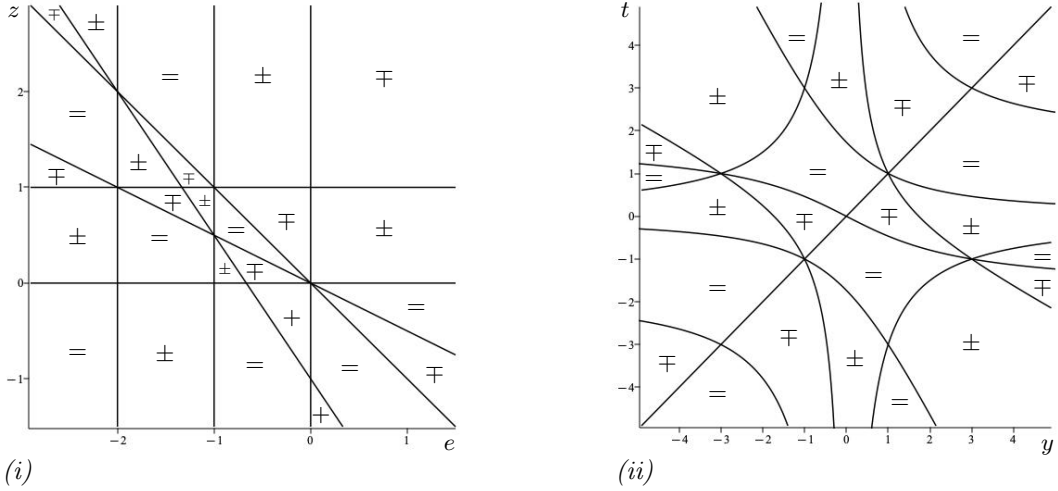

\subsection{Complete factorization for some $b,c$}

Suppose that the parameters $b,c$ have the form (\ref{eq:bcst}) for some \mbox{$z=s\in\QQ$} and $e\in\QQ$.
Considering (\ref{eq:hpg3}) as a cubic polynomial in $z$, it then has the root $z=s$. 
For generic values of $e,s$, the other two roots solve the quadratic equation
\begin{align} \label{eq:hpg3z2}
\! \frac{e^2-2es-2s^2+3e+2}{e+2s-2}\,z^2-(2e+s+1)sz 
+\frac{e^2+3es+2s^2-2e-2}{e-2s+2}\,s^2 & =0.
\end{align}
We are interested when those other roots are in $\QQ$ as well. 
For that, the discriminant
\begin{equation}
\frac{3s^2(s-1)^2(e+2s+2)(3e+2s+2)}{(e+2s-2)\,(e-2s+2)}.
\end{equation}
must be a full square. We apply the birational transformation
\begin{equation}
e=2 \,\frac{\xi-1}{\eta-\xi},\qquad s=\frac{\eta+1}{\eta-\xi},
\end{equation}
and the discriminant becomes
\begin{equation}
\frac{3\eta(1-\xi-\eta)(\xi+1)^2(\eta+1)^2}{\xi\,(\eta-\xi)^4}.
\end{equation}
The values $\xi,\eta$ should therefore determine a $\QQ$-rational point on the cubic surface
\begin{equation}
3\varrho^2=\xi\eta(1-\xi-\eta).
\end{equation}
This cubic surface has three singularities $(\xi,\eta,\varrho)\in\{(0,0,0),(0,1,0),(0,0,1)\}$.
By intersecting this surface with the pencil of 
lines $\xi=t\varrho$, $\eta=y\varrho$ through the singularity $(0,0,0)$ we obtain this birational parametrization by $(t,y)\in\PP^2$:
\begin{equation}
\xi=\frac{ty-3}{y(y+t)}, \qquad \eta=\frac{ty-3}{t(y+t)}, \qquad \varrho=\frac{ty-3}{ty (y+t)}.
\end{equation}
This translates to
\begin{equation} \label{eq:stuv}
e=\frac{2t\,(y^2+3)}{(t-y)(ty-3)}, \qquad s=\frac{y\,(t^2+2ty-3)}{(y-t)(ty-3)},
\end{equation}
and ultimately to
\begin{equation} \label{eq:bcuv}
b=-\frac{(t^2+3)(y^2+3)\,(ty+3)}{3(t^2+2ty-3)(y^2+2ty-3)}, \qquad 
c=\frac{(y^2+3)(t^2y-3y-6t)}{3(y-t)(y^2+2ty-3)}.
\end{equation}
This parametrizes the cases when the cubic polynomial (\ref{eq:hpg3}) factors completely.
The special case $y=-t$ gives the following neat family:
\begin{equation}  \label{eq:bcwt} \textstyle
b =\frac13\,t^2-1 , \qquad c=-\frac16\,t^2-\frac12.
\end{equation}
The algebraic map (\ref{eq:stuv}) is 2-to-1 generically as the central symmetry $(y,t)\mapsto(-y,-t)$ keeps $e,s$ invariant
and permutes the roots of (\ref{eq:hpg3z2}). 
Up to this symmetry, the inverse map to (\ref{eq:stuv}) is
\begin{equation}
y=\sqrt{\frac{3(e+2s+2)(e-2s+2)}{(e+2s-2)(3e+2s+2)}}, \qquad
t=\sqrt{\frac{3(e+2s-2)(e-2s+2)}{(e+2s+2)(3e+2s+2)}}.
\end{equation}
There is no point $(y,t)\in\RR^2$ where both $b$ and $c$ are positive by the parametrization (\ref{eq:bcuv}); 
see Figure \ref{fig:sregions}\refpart{ii}, and note that the only points $(y,t)\in\{(3,-1),(-3,1)\}$ where both $b,c$ 
are undetermined by (\ref{eq:bcuv}) give $(e,s)=(-1,-3)$ in (\ref{eq:stuv}) and hence rejectable $(b,c)=(0,-1)$ in (\ref{eq:bcst}).

The following map moves $z=s$ to another root of (\ref{eq:hpg3}):
\begin{equation}
(y,t)\mapsto \left( \frac{y-3}{y+1} ,\, \frac{t-3}{t+1}\right). 
\end{equation}
By composition with the central symmetry (on either or both sides) one realizes all 6 permutations of the roots of (\ref{eq:hpg3})
by fractional-linear transformations of the parameters $y,t$.

\subsection{Elliptic fibration by $b$ or $c$}
\label{sec:h3efb}

For fixed $b\not\in\{0,-1,-2\}$, equation (\ref{eq:hpg3}) defines a cubic curve of genus 1.
Its $j$-invariant equals
\begin{equation} \label{eq:ec3j}
-\frac{27\,b^2\,(9b^2+32b+32)^3}{64\,(b+1)^3\,(b+2)^2}.
\end{equation}
We can consider $\SU_3$ as an elliptic surface \cite[III.11]{SilvermanA} over $\QQ(b)$.  
Here is an isomorphic Weierstrass form for the elliptic surface,
and a characterization of rational curves on it.
\begin{propose} \label{th:hypergell}
For $b\not\in\{0,-1,-2\}$, the cubic relation $(\ref{eq:hpg3})$
between $c$ and $z$ defines a curve of genus $1$, isomorphic to the elliptic curve
\begin{equation} \label{eq:gell}
\EE_3: \quad v^2=u^3 
+b^2\,(3u-16b-16)^2
\end{equation} 
by the isomorphism 
\begin{align}  \label{ce:isom}
c= -(b+2) \frac{v+3bu-16b(b+1)}{2v}, \qquad
z= \frac{v+(3b+4)u-16b(b+1)}{2v}.
\end{align}
The Mordell-Weil group for $(\ref{eq:gell})$ of rational points  over $\QQ(b)$ is isomorphic to $(\ZZ/3\ZZ)\times \ZZ$.
It is generated the torsion point $\big(0,16b(b+1)\big)$ and the point $\big(8b,8b(b+2)\big)$. 
\end{propose}
\begin{proof}
It is straightforward to check the isomorphism (\ref{ce:isom}). 
The inverse map is given 
\begin{align}
u=\frac{16b(b+1)(bz+2z+c)}{bc+(b+2)(3bz+2c+2)},\qquad v=\frac{32b(b+1)(b+2)}{bc+(b+2)(3bz+2c+2)}.
\end{align}
As a side note, the elliptic involution $v\mapsto -v$ corresponds to the hypergeometric symmetry (\ref{eq:hpgsym2})
under this isomorphism.

The canonical Weierstrass form $v^2=u^3+Au+B$ can be obtained by applying the shift $u\mapsto u-3b^2$ in (\ref{eq:gell}).
The discriminant  \cite[III.1]{SilvermanA} of the elliptic surface equals $\Delta=2^{18}3^3b^4(b+1)^3(b+2)^2$.
By considering $A,B,\Delta$ and the $j$-invariant in (\ref{eq:ec3j}), 
we can use \cite[Table 5.1]{SchuttShioda} and conclude that the singular fibers $b\in\{0,-1,-2,\infty\}$ 
have the Kodaira types IV, I$_3$, I$_2$, I$_3$, respectively.
By \cite[Fig.~5.1--5.2]{SchuttShioda} they have, respectively, $m_v=3,3,2,3$ irreducible components in the Kodaira-Neron model. 
By Shioda-Tate formula \cite[Corollary 6.7]{SchuttShioda}, the Mordell-Weil rank equals $10-2-7=1$ as the Neron-Severi rank $\rho$ is $10$ 
for rational elliptic surfaces \cite[\S 7.2]{SchuttShioda}. 
By \cite[Theorem 8.33]{SchuttShioda}, the Mordell-Weil group is generated by polynomial points $(u,v)$ with $\deg u\le 2$, $\deg v\le 3$.
Direct computation with undetermined coefficients gives  8 possibilities for such $u$:
\begin{align} \textstyle\!  \label{eq:ec3bv}
u\in\{0,8b,-16b,4(b+1),16(b+1),-8b(b+1),16b(b+1),\frac{16}{9}(1-b)(1+2b)\}.
\end{align}
The points with $u=0$ are flex points with the tangents \mbox{$\pm v=3bu-16b(b+1)$,}  hence they are 3-torsion points.
An investigation of the group relations between all candidate points shows that 
we can take a point with $u\in\{8b$, $16(b+1)$, $-8b(b+1)\}$ for a free generator. 
\end{proof}

Taking $b\in\QQ\setminus\{0,-1,-2\}$ in (\ref{eq:gell}) gives an elliptic curve over $\QQ$,
with the same Mordell-Weil group typically \cite[III.11]{SilvermanA}. 
But the Mordell-Weil group group over $\QQ$ could be larger than the projected set of points 
from the elliptic surface. 
For example, $b=-7$ gives the Mordell-Weil group $(\ZZ/3\ZZ)\times \ZZ^2$,
as we will consider in Example \ref{eq:pqrr113}.

It is instructive to consider several simplest sections $(u,v)$ on $\EE_3$, parametrized by $b$,
and map them to the original surface $\SU_3$ defined by 
(\ref{eq:hpg3}). Here are a few obtained non-degenerate $c$-values of low degree in $b$:
\begin{equation}
\frac{b+2}{4b-1}, \  -\frac{4b(b+2)}{5b+1}, \  -\frac{27b(b+1)}{16b^2+19b+1}, \  -\frac{(2b+1)^3}{7b^2+b+1}, \ 
-\frac{2(b-1)^3}{2b^2+17b-1}.
\end{equation}
After substituting these values into (\ref{eq:hpg3}), there is a linear factor in $z$.
That parametrizes fully the corresponding section in $(c,z)$ on $\SU_3$.

If we have a root described by (\ref{ce:isom}) of the cubic hypergeometric polynomial with fixed $c$ (and possibly $b$),
the other two roots are equal to
\begin{equation} \label{eq:z23}
z_{2,3}=\frac{1}{2v}\left(v+(3b-2)u-16b(b+1)\pm u\,\sqrt{\frac{3u}{b+1}-12}\right).
\end{equation}
To seek rationality of these roots, we may replace $u=(b+1)(w^2+12)/3$ 
and obtain a family of hyperelliptic curves $(w,v)$ of genus 2.
One may go through a limited list of Mordell-Weil points on $\SU_3$ (be it specialized, with $b=-5/2$ for example)
and check that apparently complete factorization of the cubic polynomial happens generically only in degenerate cases 
such as $c\in\{-1,-b-1\}$. We need $b$ to be expressible as in (\ref{eq:bcuv}).  
The family (\ref{eq:bcwt}) corresponds to the infinite point on the elliptic surface.

For fixed $c\not\in\{0,-1,-2\}$ equation (\ref{eq:hpg3}) has degree 6,
but it defines an isomorphic elliptic surface by the symmetry (\ref{eq:hpgsym1}) of hypergeometric polynomials.
One can take $b\leftrightarrow c$, $z\mapsto 1/z$ in Proposition \ref{th:hypergell}.

\section{Quartic hypergeometric polynomials}
\label{sec:hpg4}

A compact polynomial form of the equation $\:\hpg21{-4,\;b}{-c-3}{z}=0\;$ is
\begin{equation}  \label{eq:bc4z}
\big((bz+c)^2+3bz^2+3c\big)^2+2(bz^2+c)^2+8bcz(z-1)^2+6(bz^4+c)=0.
\end{equation}
Let us denote this surface by $\SU_4$. 
We will associate two elliptic surfaces $\EE_4$ and $\EE^*_4$ to it.

\subsection{Elliptic fibration by $z$}

For fixed $z$, the curve has genus 1 generically. 
We identify $\SU_4$ directly as an elliptic surface.
As such, it is isomorphic to
\begin{align} \label{eq:el4}
\EE_4:  \quad v^2 = u^3-20Z\,u^2+108Z^2\,u-648(Z-1)^2.
\end{align}
where  $Z=z^2-z+1$.
The $j$-invariant is rather untidy in the denominator:
\begin{equation} \label{eq:es4jd}
-\frac{2^8\,19^3\,Z^6}{3^4\,(72Z^6 - 860Z^5 + 3907Z^4 - 9608Z^3 + 13122Z^2 - 8748Z + 2187)}.
\end{equation}
An isomorphism is given by 
\begin{equation}
b=-1+\varphi_4(z,u,v), \qquad c =-1+\varphi_4\!\left(\frac1z,\frac{u}{z^2},\frac{v}{z^3}\right),
\end{equation}
with
\begin{equation}
\varphi_4(z,u,v)=\frac{2\big(u^2-4Z^2\,u+8Z^3-12Z^2\big)(1-2z)v+2W_1}
{z(1-z)\,\big(u^2-12Zu+12Z^2\big)^2},
\end{equation}
where
\begin{align}
\! W_1 = & \;(5Z-3)u^3+6\,(2Z^3-24Z^2+27Z-9)u^2 \\
& -4Z(20Z^3-159Z^2+171Z-5) u +72Z^2(Z-1)(2Z^2-12Z+9). \quad\nonumber 
\end{align}
The infinite point on (\ref{eq:el4}) is mapped to the central point $(b,c)=(-1,-1)$ among the degenerations.
The parametric expression for $bz+c$ simplifies to
\begin{equation}
bz+c=\frac{2Zv-(z+1)\big(u^2-2(8Z-3z)\,u+36(z-1)^2Z\big)}
{u^2-12Zu+12Z^2}.
\end{equation}

As we consider in the next subsection, $\EE_4$ is a rational elliptic surface.
Its Mordell-Weil group has no torsion and the maximal rank 8 for elliptic surfaces over $\CC\PP^1$,
as evident \cite[\S 7.3]{SchuttShioda} from the irreducible degree 12 denominator of the $j$-invariant (\ref{eq:es4jd}),
meaning that the surface $\EE_4$ has 12 singular fibers of Kodaira type $I_1$.
By \cite[Theorem 7.12\refpart{i}]{SchuttShioda}, 
there are 240 candidate points with a polynomial coordinate $u(z)\in\CC[z]$ of degree $\le 2$
to generate the Mordell-Weil group. Computations show that 60 of them are defined over $\QQ(z)$. They have
\begin{equation}
u \in \{6,9,18,54,9z^2-6z+9,18z^2+12z+18,2z^2+4z+6,6z^2+12\}, 
\end{equation}
or can be obtained by further applying the hypergeometric symmetries (\ref{eq:hpgsym1}), (\ref{eq:hpgsym2})--(\ref{eq:hpgsym6}),
generated by $u(z)\mapsto z^2\,u(1/z)$ and $u(z)\mapsto u(1-z)$. 
Further, the two points with $u=8z^2-8z+8$ are defined over  $\QQ(\sqrt6,z)$,
and the points with $u=0$ are defined over $\QQ(\sqrt{-2},z)$. 
Additionally, 64 points have $u\in\QQ(\sqrt6)[z]$, 64 points have $u\in\QQ(\sqrt{-2})[z]$,
and 48 points have $u\in\QQ(\sqrt6,\sqrt{-2})[z]$.
The Mordell-Weil group over $\QQ(z)$ has rank 6. It is generated by points with $u\in \{6,6z^2,6(z-1)^2,9,9z^2,9(z-1)^2\}$.
Examples of simpler  rational sections on the original surface (\ref{eq:bc4z}) have these non-degenerate values of $b$:
\begin{align} \label{eq:r6gen}
& \frac{(z+1)(z+2)}{3z\,(1-z)}, \quad  
\frac{3(z+1)(3z+2)}{z\,(1-z)},\quad  \frac{6(z+1)(3z-2)}{25z\,(1-z)},\quad 
\frac{2(2z+1)(4z-3)}{z\,(1-z)(2z^2-3)^2}, \; \nonumber \\
& \frac{(2-z)(z^2-4z+1)(z^2-6z+3)}{z\,(5z^2-2z-1)^2},\quad 
\frac{2z\,(2z+1)(6z^3+z^2-2)}{(1-z)(3z^2-2)^2}.
\end{align}
The corresponding $c\,$-coordinate is obtained from a linear factor of (\ref{eq:bc4z}) that arrises after substituting $b$.
Due to the hypergeometric symmetry (\ref{eq:hpgsym1}), 
some possible values for $c$ can be obtained by substituting $z\mapsto 1/z$ in (\ref{eq:r6gen}).

\subsection{The rational surface}

The discriminant of the elliptic surface (\ref{eq:el4}) is of degree 12 in $z$; see the denominator of (\ref{eq:es4jd}).
Therefore it is a rational surface \cite[\S 7.5]{SchuttShioda} that can be obtained from a pencil of cubic curves in $\PP^2$
(linearly parametrized by $z$) by blowing up the 9 intersection points of the pencil. 
Rather equivalently \cite{Schicho05}, it is a Del Pezzo surface
in the weighted projective space with \mbox{weights$(z:1:u:v)=(1:1:2:3)$,} of degree 1.

The surface $S_4$ is birational to the elliptic surface (\ref{eq:el4}) in the Weierstrass form,
hence it can also be obtained from the pencil of cubic curves in $\PP^2$.
To obtain a rational parametrization of $S_4$, we first notice the simpler defining equation
\begin{equation}
f\,(3 f+6(z^2+z+1)+8e z+6e^2+14e)+e(e+1)(e+2)(e+3)=0
\end{equation}
in the coordinates
\begin{equation}
e=bz+c,\qquad f = bz(z-1).
\end{equation} 
The new equation defines a Del Pezzo surface
in the weighted projective space with weights$(z:e:1:f)=(1:1:1:2)$, of degree 2.
It contains 4 lines in the hyperplane $f=0$. Choosing the line $f=e=0$ to blow down,
we apply a standard step (called {\em unprojection} in \cite[\S 5]{Schicho05}) in resolving Del Pezzo surfaces
by introducing the coordinate
\begin{equation}
w=\frac{f+2(z^2+z+1)}{e}
\end{equation} 
We can eliminate $f$ straightaway and obtain the following non-singular cubic surface in $\PP^3$:
\begin{equation}
(ew-2z^2-2z-2)(3w+8z+6e+14)+(e+1)(e+2)(e+3)=0.
\end{equation}
Subsequently, we can blow down one of the lines on the plane $3w+8z+6e+14=0$,
or apply a classical parametrization recipe \cite[\S 10.5.3]{SchuttShioda} using two skew lines on the cubic surface, 
say  $e=3w+8z=-2$ and $w=e+1=-z-1$. Eventually, a suitable simplified cubic pencil 
is defined by $U+zV=0$, where 
\begin{align} \label{eq:cp4}
U= &\; 2t^2y - 6t^2 + 4ty - 3y^2 + 3y, \\
V= &\; ty^2+4t^2-2ty+3t-6y.
\end{align}
A parametrization of $S_4$ is given by $z=-U/V$ and
\begin{align} \label{eq:bcp4}
b= &\; \frac{3(y^2-2t+y)(ty^2-8t^2+4ty-3y^2+3t+3y)}{U\,(U+V)}, \\ 
\label{eq:bcp4a}
c= &\; -\frac{6(4t^2+2t-3y-3)(t^2y-t^2+ty-y^2)}{V\,(U+V)}.
\end{align}
The inverse map is given by
\begin{align}
t= &\; \frac{3(b+1)}{2}-\frac{3(c+1)(c+2)}{2bz}+\frac{3(b+c+1)(b+c+2)}{2b\,(z-1)}, \\
y= &\, -c+\frac{(b+2)(b+3)z}{c}+\frac{(b+c+2)(b+c+3)z}{c\,(z-1)}.
\end{align}
Computations with two versions of $(t,y)$ in (\ref{eq:cp4})--(\ref{eq:bcp4}) 
show that the hypergeometric symmetries (\ref{eq:hpgsym1}), (\ref{eq:hpgsym2})--(\ref{eq:hpgsym6})
are realized by non-linear Cremona \cite{Cremona02} transformations of the parameters $t,y$. 
For example,  (\ref{eq:hpgsym1}) is realized by
\begin{equation}
(t,y)\mapsto \left( \frac{t(y^2-ty^2-3t+3y)}{2(t^2y-t^2+ty-y^2)},\frac{(3-2t)(ty^2-y^2+3t-3y)}{ty^2-8t^2+4ty-3y^2+3t+3y} \right) \! ,
\end{equation}
and  (\ref{eq:hpgsym2}) is realized by
\begin{equation}
(t,y)\mapsto \left( t,\,\frac{ty-4t^2-3t+3y}{ty-t-3} \right).
\end{equation}

\subsection{The fibration by $b$}
\label{sec:hpg4fb}

Equation (\ref{eq:bc4z}) with fixed $b$ defines a quartic curve of the generic genus 3.
We can reduce the genus to 1 by factoring out the hypergeometric symmetry (\ref{eq:hpgsym2}).
This gives an elliptic surface isomorphic to 
\begin{equation} \label{eq:ec4sy}
\EE^*_4:  \quad v^2 = u\,\big(u^2-4b(5b+9)u+108\,b(b+1)^2(b+2)\big).
\end{equation}
The symmetry invariants are parametrized as follows:
\begin{align} \label{eq:h4fbz}
z(1-z)= & -6\,\frac{v+2(2b+3)u-36\,b(b+1)(b+2)}{W_2}, \\ 
\label{eq:h4fbc}
c\,(b+c+3) = & -216\,\frac{b\,(b+1)(b+2)(b+3)^2}{W_2}, \\  
\! cz\!+\!(b\!+\!c\!+\!3)(z\!-\!1) = & -(b+3)\,
\frac{2(4b\!+\!3)v+u^2\!-\!4b(b\!+\!3)u+108b(b\!+\!1)(b\!+\!2)(b\!+\!5)}{W_2}. 
\end{align}
where $W_2=8bv+u^2-4b(b+9)u+108\,b(b+1)(b+2)(b+9)$. 
The inverse projection is given by 
\begin{align}
u= &\; -6b(b+1)(b+2)\frac{(bz+c+3z)^2}{c\,(b+c+3)}, \\
v= &\, 12b\,(b+1)(b+2)(b+3)\frac{(bz+c+3z)(2bz+2c+3)}{c\,(b+c+3)}.
\end{align}
Rational points on the quartic curves can be found by trying to lift from the rational fibers on the elliptic surface.
Its Mordell-Weil group can be determined similarly as in the proof of Theorem \ref{th:hypergell}.
The discriminant of the elliptic surface equals $2^{13}3^6b^3(b+1)^4(b+2)^2(b+3)^3$, and the $j$-invariant equals 
\begin{equation}
-\frac{32\,(19b^3+36b^2-81b-162)^3}{729\,(b+1)^4(b+2)^2(b+3)^3}.
\end{equation}
Using \cite[Table 5.1]{SchuttShioda} we conclude that the singular fibers $b\in\{0,-1,-2,-3\}$ 
have the Kodaira types III, I$_4$, I$_2$, I$_3$, respectively. They have, respectively, $m_v=2,4,2,3$ irreducible components.
The Mordell-Weil rank equals $10-2-7=1$. The candidates for the generators  have
\begin{align} 
u\in  \{ & \textstyle 0, \; 9(b+1)^2, \; \frac94(b-1)^2, \; 6b(b+1), \; 54b(b+1), \; 12b(b+2),\\
& 2(b+1)(b+2),\; 18(b+1)(b+2)\}. \nonumber
\end{align}
The Mordell-Weil group is isomorphic to $(\ZZ/2\ZZ)\times\ZZ$.
The group is generated by the 2-torsion point $(0,0)$ and a point having $u\in\{6b(b+1),18(b+1)(b+2)\}$. 
The point
\begin{equation} \label{eq:h4eg}
\big( 6b(b+1),12b(b+1)(b+3) \big)
\end{equation}
can be taken as a free generator.

The discriminant of the quadratic equation (\ref{eq:h4fbz}) for $z$ equals
\begin{equation}
\frac{(v+4bu+12u)^2}{(v+4bu)^2+864b(b+1)(b+2)u}
\end{equation}
on $\EE^*_4$. 
To find rational sections on the genus 3 curve, we need the denominator 
\begin{equation} \label{eq:h4tbs}
(v+4bu)^2+864b(b+1)(b+2)u
\end{equation}
to be a full square on $\EE^*_4$. 
The point $(u,v)=0$ leads to 
\begin{equation}
c=-\frac{b+3}2\left(1\pm \sqrt{\frac{b+1}{b+9}}\right), \qquad z=\frac12\pm \frac12\,\sqrt{\frac{b+1}{b+9}}.
\end{equation}
A rational section is obtained after the base change $b=(9\zeta^2-1)/(1-\zeta^2)$.


\section{Belyi maps}

Recall that Belyi maps are algebraic coverings $\varphi:C\to\PP^1$ that branch in the 3 fibers
$\varphi\in\{0,1,\infty\}$. We will consider Belyi of genus 0 only. 
The following characteristic property of Belyi maps of genus 0 follows from the Riemann-Hurwitz formula
\begin{lemma} \label{th:np}
A Belyi map 
of genus $0$ and degree $d$ has exactly $d+2$ distinct points in the $3$ fibers $\{0,1,\infty\}$.
\end{lemma}
As summed up in (\ref{eq:tp}), 
the considered maps (\ref{eq:g11hm}) and (\ref{eq:g2hm})
satisfy the condition of this lemma. The degree of (\ref{eq:g11hm}) 
equals
\[
d=\max(|p|,|q|,|p+q|,m|r|,|p+mr|,|q+mr|,|p+q+mr|),
\]
while the degree of (\ref{eq:g2hm}) equals $\max(2|p|,m|r|,|2p+mr|)$.
Without loss of generality, we may assume $\max(|p|,|q|\le|p+q+mr|$. 

\subsection{Easy maps}

The simplest Belyi maps are the power functions $\varphi(x)=x^p$. 
They have just 2 points in the two fibers $\varphi=0$ and $\varphi=\infty$.
Our considered rational maps can  be viewed as a close neighborhood of this exemplar
in the landscape of Belyi maps.
\begin{example} \rm
The Belyi maps with exactly 3 points in two fibers are easy to find \cite{LandoZvonkin}.
They necessarily have two points with some branchong orders $k,q$ in one fiber (say $\varphi=0$)
and one point (say $x=\infty$) of branching order $k+q$ in other fiber (say $\varphi=\infty$).
There must be then $k+q-1$ distinct points in third fiber $\varphi=1$ by Lemma \ref{th:np},
hence there we have exactly one branching point, of order 2.
The most compact expression for the Belyi map is obtained after choosing
this branching point to be $x=0$:
\begin{align} \label{eq:1mf}
\varphi(x)= &\; (1-q x)^p(1+px)^q \\
 = &\; 1+O(x^2). 
\end{align}
This form matches the case $m=0$, $\lambda=p/q$ of (\ref{eq:g11hm}), 
where we take $p,q$ to be positive, and rescale $x\mapsto x/q$. 
\end{example}

\begin{example} \rm \label{ex:1m1}
Consider Belyi maps of the more general form
\begin{align} \label{eq:1ma}
\varphi(x)= &\; (1-qx)^p\,F_m(x)^q \\
 \label{eq:1mb} 
 = &\; 1+O(x^{m+1}), 
\end{align}
where $F_m(x)$ is a polynomial or degree $m$. 
If $p,q$ are positive integers, then these maps have one point of order $p$ and $m$ points of order $q$ 
above $\varphi=0$, a single point of order $p+mq$ above $\varphi=\infty$, 
and a branching point of order $m+1$ above $\varphi=1$.
The other points above $\varphi=1$ are non-branching by  Lemma \ref{th:np}.
By the equality of (\ref{eq:1ma}) and (\ref{eq:1mb}) we have
\begin{equation} \label{eq:1mc}
F_m(x)=(1-qx)^{-p/q} \quad \mbox{mod}\ x^{m+1}.
\end{equation}
This means that the polynomial $F_m(x)$ equals the truncated Taylor series of $(1-qx)^{-p/q}$ at $x=0$.
Explicitly,
\begin{align} \label{eq:1md}
F_m(x) = \sum_{k=0}^m \frac{q^k\big(p/q\big)_k}{k!}\,x^k.
\end{align}
Following our interpretation of hypergeometric polynomials
(\ref{eq:degenpo}), we can write
\begin{equation} \label{eq:truncp}
F_m(x)=\hpg21{-m,\,\frac{p}q\,}{-m}{qx}.
\end{equation}
\end{example}

\begin{remark} \rm
We may allow $p$ or $q$ to be negative integers in (\ref{eq:1ma}). That just means that the corresponding point(s) are 
above $\varphi=\infty$ rather than above $\varphi=0$. Formulas (\ref{eq:1mc}) still hold then.
The point $x=\infty$ remains above $\varphi=\infty$ if $p+mq>0$, 
and it moves to the fiber $\varphi=0$ if $p+mq<0$.
\end{remark}
\begin{remark} \label{rm:pfp} \rm
If we take $p=0$ or $p+mq=0$, then we have $m+1$ (rather than $m+2$) points in the two fibers $\varphi\in\{0,\infty\}$.
The branching order of $x=0$ can be at most $m$ then, and (\ref{eq:1mc}) should be modified to $F_m(x)=1+O(x^m)$.
We obtain the Belyi maps $\varphi(x)=(1-x^{m})^q$, after scaling $x$ additionally.
\end{remark}

\begin{example} \rm  \label{ex:pqr111}
Now we look at the Belyi maps of the form
\begin{equation} \label{eq:ab1b}
\varphi(x)=(1-x)^p(1-\lambda x)^q(1-\mu x)^r,
\end{equation}
with $\varphi(x)=1+O(x^3)$. This is the case $m=1$ of (\ref{eq:g11hm}).
The series expansion of (\ref{eq:ab1b}) starts with
\begin{equation}
1-(p+\lambda q+\mu r)\,x+\frac12 \left( (p+\lambda q+\mu r)^2- p-\lambda^2 q-\mu^2 r \right) x^2+\ldots.
\end{equation} 
The coefficients to $x,x^2$ must vanish. 
Eliminating $\lambda$ we get a quadratic equation for $\mu$. 
Its discriminat equals $-pqr(p+q+r)$, hence
the Belyi maps are expressed with $\sigma=\pm\sqrt{-pqr(p+q+r)}$. Explicitly,
\begin{equation}
\varphi(x)=\big(1-x)^p\,\left(1+\frac{pq+\sigma}{q(q+r)}\,x\right)^q
\,\left(1+\frac{pr-\sigma}{r(q+r)} \,x\right)^r.
\end{equation}
These Belyi maps with the definition field $\QQ(\sigma)$ are obtained in \cite[Example 2.2.25]{LandoZvonkin}. 
If $p,q,r$ are positive integers, the definition field is an imaginary quadratic extension of $\QQ$. 
But the Belyi maps could be defined over $\QQ$ if both positive and negative powers are prescribed.
For example, here is a nice family of paired Belyi maps with $p={n+1\choose 2}$, $q={n\choose 2}$, $r=-1$:
\begin{align}
\frac{(1+x)^p\,(1-x)^q}{1+nx},\qquad
\frac{\big(1+(n-2)x\big)^p\,\big(1-(n+2)x\big)^q}{1-n^2x}.
\end{align}
A general rational parametrization of the triples $(p,q,r)$ with $\rho\in\QQ$ is obtained 
by parametrizing the singular cubic surface $y^2+st(s+t+1)=0$, identifying \mbox{$s=p/r$,} \mbox{$t=q/r$.} 
Here is such a parametrization by $u=\sigma/qr$, $v=\sigma/pr$, up to simultaneous scaling of $p,q,r$:
\begin{equation}
p=u(uv+1),\quad q=v(uv+1),\quad r=-u-v. 
\end{equation}
Then $\sigma=-uv(u+v)(uv+1)$. After rescaling $x$ in the two forthcoming Belyi maps, 
we obtain the expressions
\begin{align}
\big(1-(v+1)x)^u\,\big(1+(u-1)x\big)^v
\,\big(1-(uv+1)x\big)^{-(u+v)/(uv+1)}, \\
\big(1+(v-1)x)^u\,\big(1-(u+1)x\big)^v
\,\big(1-(uv+1)x\big)^{-(u+v)/(uv+1)},
\end{align}
to be raised to a common power so to make the three powers integral. 
The cases with $\sigma=0$ fall outside the considered shape (\ref{eq:ab1b}),
and the cases \mbox{$(u+v)(u^2-1)(v^2-1)=0$} with coalescing branching points require special attention.
In the latter situation 
we have $(p+q)(q+r)(q+r)=0$.
\end{example}

\begin{remark} \rm  \label{rm:pqr111}
If $p+q=0$ in the last example, then $\sigma=\pm pr$.
Taking $\sigma=pr$ gives the trivial function $\varphi(x)=1$. 
Taking $\sigma=-pr$ gives
\begin{equation}
\varphi(x)=\frac{\big(1-(p-r)x)^p\,\left(1-2px\right)^r}{\left(1-(p+r) x\right)^{p}}.
\end{equation}
In the cases $p+r=0$ or $q+r=0$ we also have just one Belyi map analogously.
If simultaneously $p+q=p+r=0$, both candidate functions collapse to $\varphi(x)=1$.
\end{remark}

\section{Belyi maps of the form (\ref{eq:g11hm})}

Belyi maps of the shape (\ref{eq:g11hm}) generalize of Example \ref{ex:pqr111} from $m=1$. 
We assume that the numbers $p,q,r,p+q+mr$ are non-zero,
so that there are indeed $m+3$ distinct points above $\varphi=0$ and $\varphi=\infty$. 
The points $x=1$, $x=1/\lambda$, $x=\infty$ and their branching orders $p$, $q$, $-p-q-mr$ are permutable.
We also assume $p\neq q$, $p\neq q+mr$, $q\neq p+mr$, as these cases are considered in Section \ref{sc:main2}
with the smallest field of definition.

The condition (\ref{eq:bmpws}) translates into the power series relation
\begin{equation} \label{eq:g11hms}
G_m(x) =(1-x)^{-p/r}(1-\lambda x)^{-q/r}\quad \mbox{mod } x^{m+2}.
\end{equation}
The power series term with $x^{m+1}$ has to equal $0$ then. 
The polynomial $G_m$ is determined uniquely by the power series 
\begin{equation} \label{eq:pw11}
(1-x)^{-p/r}(1-\lambda x)^{-q/r}=  \sum_{k=0}^{\infty} h_kx^k
\end{equation}
truncated at the $(m+1)$-the term. 
\begin{lemma} \label{th:nexist}
No Belyi map of the form $(\ref{eq:g11hm})$, $(\ref{eq:bmpws})$  
exist when $p= -\ell_1 r$ 
and $q=-\ell_2 r$ for some positive integers $\ell_1,\ell_2$
satisfying $\ell_1+\ell_2\le m$.
\end{lemma}
\begin{proof}
The left-hand side of (\ref{eq:pw11}) is the polynomial $(1-x)^{\ell_1}(1-ax)^{\ell_2}$ under the stated conditions. 
When $\ell_1+\ell_2<m$, the polynomial $G_m$ is forced to be of smaller degree than $m$.
If $\ell_1+\ell_2=m$ then $G_m$ has the undue roots $x=1$, $x=1/\lambda$.
\end{proof}
The coefficients $h_k$ in  (\ref{eq:pw11}) can be expressed as follows:
\begin{align} \label{eq:ck}
h_k= & \,\sum_{j=0}^k \frac{(p/r)_{k-j}\,(q/r)_j\,\lambda^j}{(k-j)!\,j!}, 
\end{align}
or 
\begin{align}  \label{eq:ckh}
h_k & = \frac{(p/r)_k}{k!}\,\hpg21{-k,\,\frac{q}{r}}{1-k-\frac{p}{r}}{\lambda}  \\
& = \frac{(q/r)_k\lambda^k}{k!}\,\hpg21{-k,\,\frac{p}{r}}{1-k-\frac{q}{r}}{\frac1{\lambda}}.
\end{align}
We must have $h_{m+1}=0$, which leads to the equation (\ref{eq:g11hm21}) generically.

\subsection{Full sets of Belyi maps}

Here we consider examples of the generic case (\ref{eq:g11cond}) with $m+1$ Belyi maps.
We are interested in the cases $\lambda\not\in\{0,1\}$ and $p\neq q$. 
The degenerate value $\lambda=1$ would coalesce the points of branching order $p$ and $q$,
but this point is excluded by Lemma \ref{th:ck11}\refpart{iii}. 
By Lemma \ref{th:ck11}\refpart{ii}, the hypergeometric polynomial $h_{m+1}(\lambda)$ 
has $m+1$ distinct roots
when 
\begin{equation} \label{eq:g11cond}
\left(\,\frac{p}{r}\,\right)_{\!m+1}\neq 0, \quad
\left(\,\frac{q}{r}\,\right)_{\!m+1}\neq 0, \quad
\left(\frac{p+q}{r}\right)_{\!m+1}\neq 0.
 \end{equation}
The condition $((p+q)/r)_{m+1}\neq 0$ is equivalent to 
\begin{equation}
\left(-\frac{p+q+mr}{r}\right)_{\!m+1}\neq 0,
\end{equation}
underscoring the symmetry of the branching orders $p$, $q$, $-p-q-mr$.

\begin{example} \rm \label{ex:pq2}
In the case $m=2$ we consider the equation (\ref{eq:hpg3}) specialized 
by $b=q/r$, $c=p/r$, $z=\lambda$:
\begin{equation} \label{eq:pq112}
q(q+r)(q+2r)\lambda^3+3pq(q+r)\lambda^2+3pq(p+r)\lambda+p(p+r)(p+2r)=0.
\end{equation}
The discriminant 
\[
-108p^2q^2r^3\,(p+r)\,(q+r)(p+q+r)(p+q+2r)^2
\]
is proportional to $(p/r)_3(q/r)_3((p+q)/r)_3$,
and it indeed vanishes only when we have less than 3 distinct roots,
and those roots $\lambda\not\in\{0,1\}$ Lemma \ref{th:ck11}\refpart{ii}.

Belyi maps defined over $\QQ$ can be found using the parametrization (\ref{eq:bcst}),
identifying $q=br$, $p=cr$, $\lambda=z$. 
Polynomial Belyi maps 
defined over $\QQ$ are obtained from the + regions (\ref{eq:shb1}) in Figure \ref{fig:sregions}\refpart{i}.
For example, $e=-z=\frac15$ and  $e=-z=\frac16$ produce these Shabat polynomials, respectively:
\begin{align}
(1-5x)^2(1+x)^{20}(1-2x+9x^2)^5, \\
(1-6x)(1+x)^{20}(1-2x+6x^2)^7.
\end{align}
Their companions with the same $p,q,r$ are defined, respectively, over $\QQ(\sqrt{-35})$ and  $\QQ(\sqrt{-2})$.
There are no triples of Shabat polynomials defined over $\QQ$, as there is no point $(t,y)\in\RR^2$
giving positive values for $b=q/r$ and $c=p/r$ by Figure \ref{fig:sregions}\refpart{ii}.
The family (\ref{eq:bcwt}) with $t\in\QQ$ provides most of the triples of Belyi maps defined over $\QQ$ 
with small absolute values of $p,q,r$.
For example, $t=2$ gives the case $(p,q,r)=(2,-7,6)$ with these 3 Belyi maps defined over $\QQ$:
\begin{align}
\frac{(1-x)^2(1-2x-\frac16x^2)^6}{(1-2x)^7}, \quad \frac{(1-x)^2(1+5x+\frac{10}3x^2)^6}{(1+4x)^7}, \quad
 \frac{(1-5x)^2(1-3x-\frac{2}3x^2)^6}{(1-4x)^7}.
\end{align}
Here are a few projective ratios  $(p:q:r)$ giving 3 Belyi maps defined over $\QQ$ outside the family (\ref{eq:bcwt}):
\begin{align} \textstyle
(1:-\frac78:\frac1{19}),\, (2:-\frac15:\frac1{37}), \, (1:-\frac{26}{31}:\frac{1}{37}), \, (1:-\frac{11}{41}:\frac{1}{43}), 
\, (1:-\frac{14}{17}:\frac{1}{61}),  \nonumber \\
 \textstyle
 (7:-\frac52:\frac{4}{13}), \;  (7:-\frac{65}{19}:3), \; (19:-\frac{35}{8}:1),\; (-19:-\frac{37}{7}:\frac12). 
\end{align}

Families of Belyi maps could be obtained from sections on the elliptic siurface $\EE_3$,
particularly from the $c$-values in (\ref{eq:ec3bv}). Taking $c=(b+2)/(4b-1)$ with the reparametrization $t=4b-1$
and independent rescaling of $x$ and of the powers $p,q,r$ gives
\begin{equation}
(1-2tx)^{t+9}\,\big(1+6x\big)^{t(t+1)} \left(1-(t-3)x-(t+3)^2x^2\right)^{4t}.
\end{equation}
Similarly, taking $c=-4b(b+2)/(5b+1)$ and $t=(5b+1)/(3b)$ gives
\begin{equation}
(1-tx)^{12-8t}\,\big(1-2(t-1)x\big)^{t} \left(1-2x-\frac{2(t-2)^2}{3t-5}x^2\right)^{\!t(3t-5)}.
\end{equation}
Or one may employ the whole 2-parameter uniformizations (\ref{eq:bcst}) or (\ref{eq:bcwt}).
\end{example}

\begin{example} \label{ex:surf4} \rm
Similarly, for $m=3$ we obtain the surface $\SU_4$ in (\ref{eq:bc4z}) with $b=q/r$, $c=p/r$, $z=\lambda$.
We can use the full parametrization (\ref{eq:bcp4})--(\ref{eq:bcp4a}), or a section of the elliptic surface $\EE_4$ in (\ref{eq:el4}).
For example, the first $b$-value from (\ref{eq:r6gen}) gives the family
\begin{align}
& (1-x)^{2t(t-2)(t+2) }\,\big(1-tx\big)^{(t+1)(t+2)}  \nonumber \\
& \qquad \textstyle \times \left(1-(t+2)x-\frac{(t+1)(t+2)}3x^2-\frac{(t+1)(t-2)(t+2)}9x^3\right)^{3t(1-t)},
\end{align}
with $t=z$. Here are some values of $b,c$ with one of them a positive integer,
found by an extensive search  through the parametrization (\ref{eq:bcp4}):
\begin{align*}
&  \textstyle (5,-\frac{7}{2}),\, (5,-\frac{7}{11}),\, (5,-\frac{14}{13}),\, (13,-\frac{13}{3}),\,
(13,-\frac{26}{5}),\, (14,-\frac{7}{2}),\, (15,-\frac{45}{43}),\, (19,-\frac{28}{3}), \quad \nonumber   \\
&  \textstyle (19,-\frac{19}{7}),\, (20,-\frac{92}{11}),\, (22, -\frac{40}{7}),\, 
(23,-\frac{25}{2}),\, (24,-\frac{117}{47}),\, (45,-\frac{256}{13}),\, (45, -\frac{207}{19}),\nonumber  \\
& \textstyle (45,-\frac{47}{51}),\, (51, -\frac{442}{71}),\, (54,-\frac{135}{7}),\,  (55,-\frac{35}{2}),\, (55,-\frac{99}{4}),\,
(56,-\frac{29}{5}),\, (56,-\frac{133}{5}), \nonumber  \\
& \textstyle (68,-\frac{92}{7}),\, (68,-\frac{355}{33}),\, (68,-\frac{85}{131}), \,
(76,-\frac{143}{5}),\, (77,-\frac{79}{2}),\, (84,-\frac{192}{5}),\, (91,-\frac{188}{41}).
\end{align*}
The negative integer values of $b,c$ are considered in Example \ref{ex:degd4}.
Here are some other found values of $b,c$, excluding the values $\{-1/2,-3/2,-5/2,-7/2,-9/2\}$ 
with reference to Examples \ref{ex:ec7} and \ref{ex:ec8}:
 \begin{align*}
&  \textstyle (\frac12,-\frac{8}{3}),\, (\frac12,-\frac{7}{5}),\, (\frac72,-\frac{26}{5}),\, (\frac{17}2,-\frac{3}{10}),\, 
(\frac43,-\frac{27}{7}),\, (\frac73,-\frac{13}{9}),\, (\frac73,-\frac{26}{5}),\, (\frac{14}3,-\frac{34}{9}), \quad \nonumber   \\
&  \textstyle \nonumber  (\frac{20}3,-\frac{23}{9}),\, (\frac{29}3,-\frac{63}{5}),\, (\frac{74}{3},-\frac{37}{2}),\, 
 (\frac{11}4,-\frac{40}7),\, (\frac{15}4,-\frac{45}7),\, (\frac{15}4,-\frac{46}7),\,  (\frac45,-\frac{19}{7}), \\
&  \textstyle (\frac76,-\frac{35}{9}),\,  (\frac85,-\frac{13}{4}),\, (\frac16,-\frac{9}{22}),\, 
 (\frac56,-\frac29),\, 
  (\frac17,-\frac{1}{4}),\, (\frac17,-\frac{11}{5}),\,  (\frac37,-\frac{17}{5}),\, 
 (\frac67, -\frac{18}{5}),
\nonumber  \\
& \textstyle 
(-\frac23, -\frac{7}{5}),\, (-\frac{5}{3},-\frac29),\,
(-\frac{8}{3},-\frac38),\, (-\frac14,-\frac{7}{5}),\, (-\frac54,-\frac{8}{11}),\,  (-\frac35,-\frac{4}{7}),\,
(-\frac76,-\frac{9}{10}). 
\end{align*}
\end{example}

\subsection{Cases with fewer Belyi maps}

The number of Belyi maps (\ref{eq:g11hm}) with $\lambda\not\in\{0,1\}$ is smaller than $m+1$ when 
at least one of the conditions (\ref{eq:g11cond}) is not satisfied. Indeed, if $q= -\ell r$ for a positive integer $\ell\le m$,
the polynomial $h_{m+1}(\lambda)$ has the degree $\ell$ rather than $m+1$ as the terms with $j>\ell$ in (\ref{eq:ck}) vanish then.
If $p= -\ell r$ for a positive integer $\ell\le m$, then the first $m+1-\ell$ terms in (\ref{eq:ck}) vanish, and
\begin{equation} \label{eq:ck1a}
h_{m+1}=\frac{(-1)^\ell\,(q/r)_{m+1-\ell}}{(m+1-\ell)!}\,\lambda^{m+1-\ell}\;\hpg21{-\ell,\,\frac{q}{r}+m+1-\ell}{m+2-\ell}{\lambda}.
\end{equation}
The factor $\lambda^{m+1-\ell}$ reduces the number of relevant roots to $\ell$. 
If $p+q= -\ell r$ for a positive integer $\ell\le m$, but $(p/r)_m\neq0$ and $(q/r)_m\neq0$, 
then we apply Euler's transformation (\ref{eq:eulr})  to the hypergeometric polynomial $h_{m+1}(\lambda)$ 
and obtain
\begin{align}  \label{eq:ck1aa}
h_{m+1} =  \frac{(p/r)_{m+1}}{(m+1)!}\,(1-\lambda)^{\ell+1}\;\hpg21{-m+\ell,\,1+\frac{p}{r}}{-m-\frac{p}{r}}{\lambda}. 
\end{align}
The factor $(1-\lambda)^{\ell+1}$ reduces the number of relevant roots to $m-\ell$. 
If \mbox{$p/r\in\ZZ$} and $q/r\in\ZZ$, then these possibilities may combine, 
giving lower degree of $h_{m+1}(\lambda)$ and multiple undue roots $\lambda=0$, $\lambda=1$. 
Figure \ref{fig:regions11} depicts the regions and line segments or rays on the integer lattice for the values of $p/r$ and $q/r$ 
where the number of Belyi maps (\ref{eq:g11hm}) is fixed $\le m+1$. 
Lemma \ref{th:nexist} applies to the visible triangle inside the dark middle region.

\begin{figure}
\begin{center}
\begin{picture}(320,250)
\put(40,-2){\includegraphics[width=246pt]{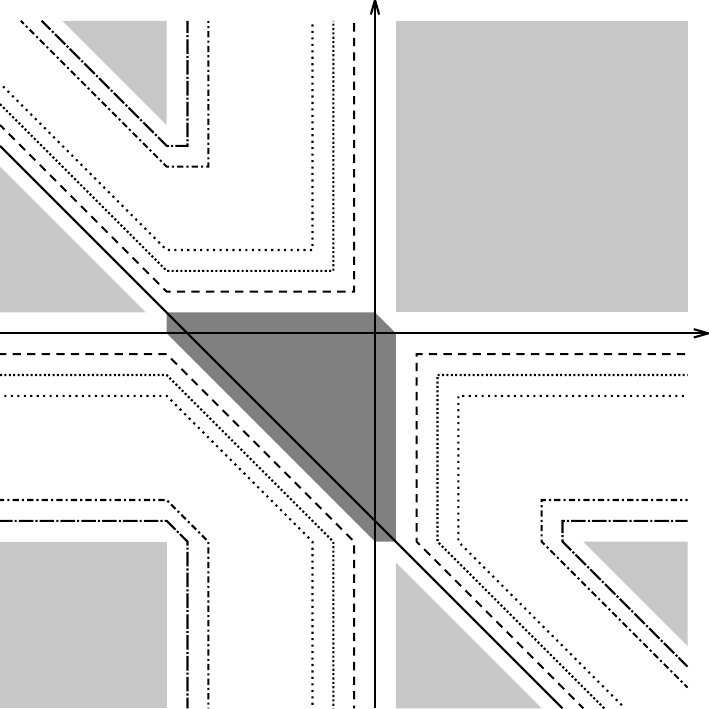}}
\put(283,135){\small $p/r$} \put(175,243){\small $q/r$} 
\put(282,119.5){\small $1$}  \put(282,111.5){\small $2$} \put(282,103.5){\small $3$}
\put(282,69){\small $m\!-\!1$} \put(282,61){\small $m$}   \put(283,38){\small $m\!+\!1$}
\put(282,3){\small $m\!-\!1$} \put(282,11){\small $m$} 
\put(33,119.5){\small $1$}  \put(33,111.5){\small $2$} \put(33,103.5){\small $3$}
\put(17,69){\small $m\!-\!1$} \put(30,61){\small $m$}   \put(17,26){\small $m\!+\!1$}
\put(33,200){\small $1$}  \put(33,208){\small $2$} \put(33,216){\small $3$} 
\put(17,160){\small $m\!+\!1$} \put(283,190){\small $m\!+\!1$} 
\put(33,192){\small $0$} \put(33,127.5){\small $0$} 
\end{picture}
\end{center}
\caption{The number of Belyi maps (\ref{eq:g11hm}) 
for integer values of $p/r$ and $q/r$.
The dark region in the middle, and the lines emanating from it, represent the cases with no Belyi maps.
The light grey regions represent $m+1$ Belyi maps; 
the dashed lines --- unique Belyi maps; the dense dotted lines --- pairs of Belyi maps;
the sparser dotted lines --- triples of maps; the two kinds of dashed-dotted lines: $(m-1)$ or $m$ maps.}
\label{fig:regions11}
\end{figure}
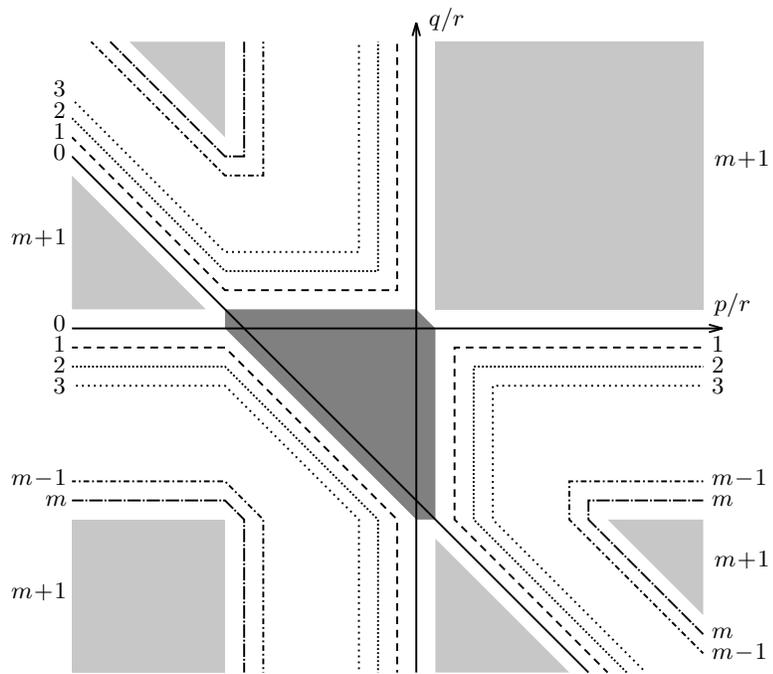

Let us use $b^*=q/r$ and $c^*=p/r$ for shorthand.
\begin{example} \rm \label{eq:pqrr111}
Let us consider the cases that have single Belyi maps. 
They are represented by some lattice points on the lines $p=-r$, $q=-r$, $p+q=(1-m)r$
or $p=2r$,  $q=2r$, $p+q=-(m+2)r$ in Figure \ref{fig:regions11}. 
It is enough to consider two cases, that is, one from both displayed triples of lines,
due to the hypergeometric symmetries (\ref{eq:hpgsym1}), (\ref{eq:hpgsym2})--(\ref{eq:hpgsym6}).
When $q=-r$ and $(c^*\!-1)_{m}\neq 0$, the hypergeometric polynomial $h_{m+1}$ 
is linear in $\lambda$. It gives $\lambda=(m+c^*)/(m+1)$. 
Following (\ref{eq:g11hms}), consider the power series
\begin{equation}
(1-x)^{-c^*}\left(1-\frac{m+c^*}{m+1}\,x\right) = 
1+\frac{1}{m+1}\sum_{k=1}^{\infty} \frac{(c^*\!-1)_k}{k!}(m+1-k)x^k.
\end{equation}
The term with $x^{m+1}$ indeed vanishes, and the earlier terms define $Q_m(x)$.

When \mbox{$p=2r$} and $(b^*)_{m+1}=0$, then (\ref{eq:ck1aa}) reads
\begin{equation}
\hpg21{-m-1,\,b^*}{-m-2}{\lambda}=(1-\lambda)^{-b^*\!-1}\,\hpg21{-1,-b^*\!-m-2}{-m-2}{\lambda}.
\end{equation}
The  hypergeometric polynomial on the right-hand side is linear, and gives
\mbox{$\lambda = (m+2)/(b^*\!+m+2)$.}
The implication for the power series of
\begin{equation} \label{eq:mcm2}
(1-x)^{-2}\left(1 -\frac{m+2}{b^*\!+m+2}\,x\right)^{-b^*}
\end{equation} 
is remarkable: the $k$-th term is divisible by $(b^*\!+1)_k$ when $m=k-1$.
Here is a revealing form of the starting terms: {\small
\begin{align}
& 1 +  \frac{m(b^*\!+2) + 4(b^*\!+1)}{b^*\!+m+2}  
+ \frac{(m-1)(b^*\!+3)\big(m\,(b^*\!+2)+9b^*\!+10\big) + 27(b^*\!+1)(b^*\!+2)}{2\,(b^*\!+m+2)^2} 
\nonumber\\
& \quad  +\frac{(m-2)(b^*\!+4)\big(m^2(b^*\!+2)(b^*\!+3)
 +  \ldots 
\big)  +  256(b^*\!+1)(b^*\!+2)(b^*\!+3)}{6\,(b^*\!+m+2)^3} +\ldots.
\end{align}}
\end{example}

\begin{example} \rm \label{ex:conic}
Here we consider the cases with two Belyi maps. 
They are represented by some lattice points on the lines $p=-2r$, $q=-2r$, $p+q=(2-m)r$
or $p=3r$,  $q=3r$, $p+q=-(m+3)r$ in Figure \ref{fig:regions11}. 
It is enough to consider two cases. When $q=-2r$ and $(c^*\!-2)_{m}\neq 0$,
the hypergeometric polynomial $h_{m+1}$ 
is quadratic in $\lambda$: 
\begin{equation} \label{eq:quadq2}
m(m+1)\lambda^2-2(m+1)(m+c^*\!-1)\lambda+(m+c^*)(m+c^*\!-1).
\end{equation}
The two solutions are
\begin{equation}
\lambda=\frac{m+c^*\!-1}{r} \pm \frac{\sqrt{(m+1)(c^*\!-1)(m+c^*\!-1)}}{m(m+1)}
\end{equation}
We have two Belyi maps are defined over $\QQ$ when
\begin{equation} 
c^*=\frac{m+1+(m-1)u^2}{m+1-u^2}.
\end{equation}
for some $u\in\QQ$. Then the $\lambda$-values are $(m+1\pm u)/(m+1-u^2)$. 
The power series then looks like this:
\begin{align}
(1-x)^{-c^*}\left(1-\frac{m+1+u}{m+1-u^2}\,x\right) =
1+\left(c^*\!-1-\frac{m+(u+1)^2}{m+1-u^2}\right)x \qquad\qquad \nonumber \\
 +\frac{1}{m}\,\sum_{k=2}^{\infty} \frac{(c^*\!-1)_{k-1}}{k!}\left(c^*\!-1-\frac{m+k(u+1)^2}{m+1-u^2}\right) \! (m+1-k)\,x^k.
\end{align}
Note that if $c^*=2$ then $u=\pm 1$, and we have just one Belyi map with $u=1$, which is a special case of (\ref{eq:mcm2}).

Like in the previous example, we apply (\ref{eq:ck1aa}) to the case $p=3r$ and $(q/r)_{m+1}=0$:
\begin{equation}
\hpg21{-m-1,\,b^*}{-m-3}{\lambda}=(1-\lambda)^{-b^*\!-2}\,\hpg21{-2,-b^*\!-m-3}{-m-3}{\lambda}.
\end{equation}
We obtain a quadratic equation analogous to (\ref{eq:quadq2}).  It has solutions $\lambda\in\QQ$ when 
\begin{equation} 
b^* =-\frac{(m+3)u^2}{m+2+u^2}
\end{equation}
for some $u\in\QQ$. Then the $\lambda$-values are $(m+2+u^2)/(m+2\pm u)$.
\end{example}

\begin{example} \rm \label{eq:pqrr113}
Now we consider the cases with larger $m$ that have three Belyi maps. 
Suppose that $q=-3r$. Then the hypergeometric polynomial (\ref{eq:ck1a}) for $h_{m+1}$ 
defines a cubic relation between $\lambda$ and $p/r$. 
This cubic relation defines a curve of genus 1, as stated in Proposition \ref{th:hypergell}.
The elliptic curve has the Mordell-Weil group isomorphic to  $(\ZZ/3\ZZ)\times \ZZ$ for general $m$.
The isomorphism (\ref{ce:isom}) has to be adjusted with 
\begin{equation} \label{eq:elladj}
b=-m-1,  \qquad c=3-m-\frac{p}{r}.
\end{equation}  
Analysis with the data base \cite{LMFDB} found the examples with $m\in\{6,11,13,17,23,25\}$ giving
the Mordell-Weil group $(\ZZ/3\ZZ)\times \ZZ^2\!$, though the cases $m\!\in\!\{14,20,21,22\}$ are not in the data base yet. 
In particular, the elliptic curve for the case $m=6$ has the label 39690.bj2 in \cite{LMFDB}. 
The equation specialized from (\ref{eq:gell}) is 
\begin{equation} \label{eq:ellc7}
y^2=x^3+49(3x+96)^2.
\end{equation}
Besides the specialized generators $(0,672)$ and $(-56,280)$ of Proposition \ref{th:hypergell},
its Mordell-Weil group has also the generator $(-48,48)$ over $\QQ$. 
This extra generator corresponds to the hypergeometric evaluation 
\begin{equation} \label{eq:kr13}
\hpg21{-7,-3}{13}{-1}=0,
\end{equation} 
and the Belyi map
\begin{equation} \label{eq:193}
\frac{1-16x+117x^2-512x^3+1463x^4-2736x^5+2907x^6}{(1-x)^{19}(1+x)^3}.
\end{equation}
It is interesting to observe that the expanded polynomial $(1-x)^{19}(1+x)^3$ does not have the terms 
with $x^7$ (as we just used) and with $x^{11}$, $x^{15}$. 
Therefore we can obtain two more Belyi maps by extending the numerator of (\ref{eq:193}) 
to polynomials of degree $m=10$ or $m=14$. This  phenomenon is typical for integer points
on our hypergeometric surfaces, as the next example suggests.
\end{example}

\begin{example} \label{ex:degd4} \rm
For $q=-4r$, the parametrization (\ref{eq:bcp4})--(\ref{eq:bcp4a}) has to be adjusted by (\ref{eq:elladj}).
We seek negative values of $b$ or $c,-b-c-m$ 
due to the hypergeometric symmetries (\ref{eq:hpgsym1}), (\ref{eq:hpgsym2})--(\ref{eq:hpgsym6}). 
Here are some cases for $(b,c)$ found:
\begin{align}
&  \textstyle (-5,-\frac{4}{7}),\, (-7,-\frac{7}{5}),\, (-9,-\frac{9}{2}),\, (-14,-\frac{26}{5}),\,  
(-14,-\frac{22}{19}),\,  (-19,-\frac{3}{20}), \quad \nonumber \\
&  \textstyle  (-19,-\frac{17}{31}),\,  (-22,-\frac{5}{31}),\, (-22,-\frac{57}{107}),\, (-24, -\frac{77}4),\,
(-24,-\frac{33}{10}),  \nonumber \\
&  \textstyle (-33,-\frac{33}{7}),\, (-36,-\frac{20}{3}),\, (-43,-\frac{43}{2}),\, (-43,-\frac{168}{127}),\, (-54,-\frac{39}{85}), \nonumber  \\
&  \textstyle (-54,-\frac{51}{233}),\, (-59,-\frac{19}{85}),\, (-59,-\frac{63}{139}),\, (-65,-\frac{52}{7}),\, (-68,-\frac{871}{19}), \nonumber \\
&  \textstyle   (-71,-\frac{180}{19}),\, (-78,-\frac{22}{3}),\, (-78,-\frac{190}7),\, (-79,-\frac{247}{5}),\,  
(-80,-\frac{231}{23}).
\end{align}
Of particular interest are the points with both $b,c$ negative integers.
Found instances are
\begin{align}
& (-7,-10),\, (-14,-66),\, (-28,-65),\, (-14,-52),\, (-30,-36),\, (-22,-35), \qquad \nonumber \\
& (-63,-64),\, (-41,-247),\, (-78,-210),\, (-115,-210),\, (-247,-780), \nonumber \\ 
& (-341,-715),\, (-901,-1856),\, (-1730,-8478),\, (-2795,-16512). 
\end{align}
They give two Belyi maps though for different cases of $m,p/r$.
For example, the first instance gives 
the hypergeometric evaluations 
\begin{equation}   \label{eq:kr74}
\hpg21{-7,-4}{7}{-1}=0, \qquad \hpg21{-10,-4}{4}{-1}=0.
\end{equation}
They both lead to consider the polynomial
\begin{align}
& (1-x)^4(1+x)^{13} = 1+9x+32x^2+48x^3-12x^4-156x^5-208x^6 \\
& \quad +286x^8+286x^9-208x^{11}-156x^{12}-12x^{13}+32x^{15}+9x^{16}+x^{17} \qquad \nonumber
\end{align}
as a power series without the terms $x^7$ or $x^{10}$. 
\end{example}

\begin{remark} \rm
The cases like (\ref{eq:kr13}), (\ref{eq:kr74}) of hypergeometric equations with integer parameters 
can be expressed in terms of {\em Krawtchouk polynomials} \cite[\S 9.11]{KS10}:
\begin{align}
K_n(x;p,N) = \hpg21{-n,-x}{-N}{\,\frac1p\,}.
\end{align}
Following (\ref{eq:hpgsym2}), hypergeometric polynomials in (\ref{eq:kr13}), (\ref{eq:kr74}) are identified by
\begin{align}
\hpg21{-m,-n}{M}{1-\frac1p} =
 \frac{(M+m)_n}{(M)_n} \, K_n(m;p,M+m+n-1).
\end{align}
Finding integer roots of Krawtchouk polynomials is an active field of research \cite{KraLit96}, \cite[\S 7.2]{StrodW99}
with special implications for graph and coding theories \cite{Habs01}, algebraic geometry \cite{CilSols01}, 
Pad\'e approximations \cite[\S 2.2]{ZhedPade}. 
\end{remark}

\section{Belyi maps of the form (\ref{eq:g2hm})}
\label{sc:main2}

The form (\ref{eq:g2hm}) of Belyi maps 
is the compacted case $q=p$ of the form (\ref{eq:g11hm}), 
where $(1-x)^p(1-\lambda x)^p$ is replaced by the aggregate power $(1+\alpha x+\beta x^2)^p$. 
This grouping of points with the same branching order is routinely used to simplify maximally
the {\em field of definition} of Belyi maps \cite{LandoZvonkin}.
Similarly, here we are not interested in the case $r=p$. 
These maps correspond to the cases $(1-x^{m+2})^p$ of Remark \ref{rm:pfp},
with a quadratic factor of $1-x^{m+2}$ distinguished.

Let us denote $H_2(x)=1+\alpha x+\beta x^2$. 
The condition (\ref{eq:bmpws}) translates into the power series relation
\begin{equation} \label{eq:g2hmz}
G_m(x) = H_2(x)^{-p/r}\quad \mbox{mod } x^{m+2}.
\end{equation}
The power series term with $x^{m+1}$ has to equal $0$.
The polynomial $G_m$ is determined uniquely by the power series 
\begin{equation} \label{eq:pw2}
H_2(x)^{-p/r}=  \sum_{k=0}^{\infty} g_kx^k
\end{equation}
truncated at the $(m+1)$-the term. The power series of 
can be computed by expanding $(-\alpha-\beta x)^j$ in 
\begin{align} \label{eq:q2pw}
H_2(x)^{-p/r}=&\, \sum_{j=0}^{\infty} \,\frac{1}{j!}\,\Big(\,\frac{p}{r}\,\Big)_{\!j}\,(-\alpha-\beta x)^jx^j.
\end{align}
Explicitly,
\begin{equation} \label{eq:gk}
g_k=\sum_{j=0}^{\lfloor k/2 \rfloor} \frac{(-1)^{k-j}(p/r)_{k-j}\,\alpha^{k-2j}\beta^j}{(k-2j)!\,j!}. 
\end{equation}
If $(p/r)_k\neq 0$ and $\alpha\neq 0$, we have the hypergeometric expression
\begin{equation} \label{eq:q2pwhpg}
g_k=\frac{(-\alpha)^k}{k!}\,\Big(\,\frac{p}{r}\,\Big)_{\!k}\;\hpg21{-\frac{k}2,-\frac{k-1}2}{1-k-\frac{p}{r}}{\frac{4\beta}{\alpha^2} }.
\end{equation}
Distinct Belyi maps correspond to the roots of $g_{m+1}(\alpha\!:\!\beta)$, 
identified up to the $(1\!:\!2)$-weighted homogeneous action of scaling $x$. 
There is a Belyi map with \mbox{$\alpha=0$} when $m$ is even.
It is obtained from Example \ref{ex:1m1} after the substitution $x\mapsto x^2$, $q=r$
(with $q,m$ over there equal to the current $r,m/2$).

The following degeneratated cases are encountered:
\begin{itemize}
\item There are no Belyi maps when $p=-\ell r$ for a positive integer \mbox{$\ell\le \lfloor m/2 \rfloor$,} 
because then $G_m(x)=H_2(x)^j$, and possibly of lesser degree than $m$.
\item If $p=-\ell r$ for a positive integer satisfying \mbox{$\lfloor m/2 \rfloor<\ell\le m$,} 
then the first $m+1-\ell$ terms for the sum (\ref{eq:gk}) for $g_{m+1}$ are zero,
and $g_{m+1}$ has the degenerate root $\beta=0$ of that multiplicity. 
After shifting the summation index by $k-\ell$,
\begin{equation} \label{eq:q2pwhp}
g_k=\frac{\alpha^{2\ell-k}\beta^{k-\ell}\,\ell!}{(k-\ell)!\,(2\ell-k)!}
\,\hpg21{\frac{k}2-\ell,\frac{k+1}2-\ell}{1+k-\ell}{\frac{4\beta}{\alpha^2}}.
\end{equation}
There are $\ell-\lceil m/2 \rceil$ proper Belyi maps then, including the $\alpha=0$ case for even $m$.
\item If $p=-\ell r/2$ for a positive odd integer $\ell\le m$, then 
\begin{equation}  \label{eq:pqj2}  
g_{m+1}=\frac{(-\alpha)^{m+1}\big(-\frac{\ell}{2}\big)_{\!m+1}}{(m+1)!} 
\left(1-\frac{4\beta}{\alpha^2}\right)^{\!\frac{\ell+1}{2}}
\hpg21{\frac{\ell-m}2,\frac{\ell+1-m}2}{\frac{\ell}2-m}{\frac{4\beta}{\alpha^2}}.
\end{equation}
by Euler's formula (\ref{eq:eulr}).  
The root $\beta=\alpha^2/4$ corresponds to the degeneration of $H_2(x)$ to a full square.
The transformed hypergeometric sum can be identified as $g_{m-\ell}$ with the substituted $p/r=1+\ell/2$. 
There are then $\lceil (m-\ell)/2 \rceil$ Belyi maps, including the $\alpha=0$ case for even $m$. 
In particular, there are no Belyi maps when the odd $\ell=m$,
and there is only the map with $\alpha=0$ when the odd $\ell=m-1$.
\end{itemize}
Otherwise, that is when $p\neq -\ell r$ for all positive $\ell\le m$
and $p\neq -\ell r/2$ for all positive odd $\ell\le m$, the $\hpgo21$-factor 
in (\ref{eq:q2pwhpg}) for $g_{m+1}$ gives \mbox{$\lceil m/2 \rceil$} distinct values with \mbox{$\beta\not\in\{0,\alpha^2/4\}$}
by Lemma \ref{th:ck2}.  Together with $\alpha=0$, in total we then have \mbox{$\lceil (m+1)/2 \rceil$} Belyi maps. 

\subsection{Full sets of Belyi maps}

Let us enumerate the following examples by the value of $m$. We are principally interested in Belyi maps defined over $\QQ$. 
\setcounter{theorem}{1}

\begin{example} \rm
For $m=2$, there is a Belyi map defined with $\alpha=0$. When $p\neq -r/2$, there is other map.
We can scale $x$ to obtain
\begin{align}
H_2(x)= &\; 1+rx+\frc16r(p+2r)\,x^2, \\
G_2(x)= &\; 1-px+\frc16p(2p+r)\,x^2. \nonumber
\end{align}
\end{example}

\begin{example} \rm
For $m=3$, the Belyi maps are defined with the radical \mbox{$R_3=\pm\sqrt{3(p+2r)(2p+3r)}$}:
\begin{align}
H_2(x)= &\; 1+rx+\frc16r\big(3p+6r+R_3\big)\,x^2, \\
G_3(x)= &\; 1-px+\frc16p\big(3r+R_3\big)\,x^2+\frc16p(p+r)\big(2p+4r+R_3\big)\,x^3. \nonumber
\end{align}
There are no Belyi maps for $p/r\in\{-1,-2,-\frac32\}$. 
For $p=-r/2$ or $p=-3r$ we have single Belyi maps, respectively:
\begin{align}
\frac{(1+2x+5x^2)^p}{(1+x+2x^2-2x^3)^{2p}},\qquad \frac{(1+3x-5x^3)^{r}}{(1+x-x^2)^{3r}}.
\end{align}
The cases when $R_3\in\QQ$ are parametrized as
\begin{align}
\frac{p}{r}=\frac{9-2s^2}{s^2-6}, \qquad s\in\QQ.
\end{align}
\end{example} 
\begin{example} \rm
For $m=4$, the Belyi maps with $\alpha\neq 0$ are defined with the radical \mbox{$R_4=\pm\sqrt{5(p+3r)(2p+3r)}$}:
\begin{align}
H_2(x)= &\; 1+rx+\frc1{30} r\big(5p+15r+R_4\big)\,x^2, \\
G_4(x)= &\; 1-px+\frc1{30}p\big(10p-R_4\big)\,x^2+\frc1{30}p(p+r)\big(5r+R_4\big)\,x^3 \quad \nonumber \\
& \, -\frc1{180}p(p+r)(2p+3r)\big(2p+6r+R_4\big)\,x^4. \nonumber
\end{align}
There are no Belyi maps for $p/r\in\{-1,-2,-3,-\frac32\}$. 
For $p=-r/2$ or $p=-4r$ we have single Belyi maps, respectively:
\begin{align}
\frac{(1+6x+21x^2)^p}{(1+3x+6x^2-18x^3+36x^4)^{2p}}, \quad \frac{(1+12x+42x^2-189x^4)^{r}}{(1+3x-3x^2)^{4q}}.
\end{align}
The cases when $R_4\in\QQ$ are parametrized as
\begin{align}
\frac{p}{r}=\frac{15-3s^2}{2s^2-5}, \qquad s\in\QQ.
\end{align}
\end{example}

\begin{example} \rm \label{ex:ec5}
The case $m=5$ is considered in \cite[\S 2.2.4.3]{LandoZvonkin}.
The cubic hypergeometric relation 
\begin{align} \label{eq:hec5}
\hpg21{-3,-\frac{5}2}{-\frac{p}{r}-5\,}{\frac{4\beta}{\alpha^2\!}}=0
\end{align}
can be analyzed within the context of Section \ref{sec:h3efb} with $b=-5/2$, $c=3+p/r$, and $z=4\beta/\alpha^2$.
We combine Theorem \ref{th:hypergell} 
with a simple transformation 
\begin{equation}. \label{eq:b52uvs}
(u,v)\mapsto \left( \frac{u-75}{4}, \, \frac{v}{8} \right).
\end{equation}
The obtained  elliptic curve is the same as in \cite{LandoZvonkin}:
\begin{align}  \label{eq:ecm5}
\EE_5:\quad v^2=u^3-2475u-5850.
\end{align}
The isomorphism obtained from (\ref{ce:isom}) is simpler than in \cite[pg.~105]{LandoZvonkin}:
\begin{align}  \label{eq:ecp5}
\frac{p}{q}=\frac{645-15u-11v}{4v},\qquad  z=\frac{45-7u+v}{2v}.
\end{align}
The data base of elliptic curves \cite{LMFDB} identifies this curve by the label 4050.y2,
and confirms 
that the Mordell-Weil group of $E_1$ is isomorphic to $\ZZ\times(\ZZ/3\ZZ)$.
Any $\QQ$-rational point on $E_1$ can be expressed using the addition law  on the elliptic curve  as
\begin{equation} \label{eq:mw1}
n\,(-5,80)+\varepsilon\,(75,480),\qquad \mbox{with } n\in\ZZ, \; \varepsilon\in\{0,1,-1\}.
\end{equation}
Several of the rational points correspond to the degenerate cases $p/r\in\left\{0,-3,-5/2 \right\}$
of no Belyi maps, $p/r\in\left\{-4,-3/2\right\}$ of single maps, and $p/r\in\left\{-5,-1/2\right\}$ of coupled maps.
The elliptic involution represents the hypergeometric identity (\ref{eq:hpgsym2}), and acts as
\[
\left( \frac{p}{r},\,z \right) \mapsto \left( -\frac{p}{r}-\frac{11}2, \, 1-z \right).
\]
There are rational points with $p/r\in\{-11/2,-11/4\}$. 
Here are some other values of $p/r$ that give Belyi maps defined over $\QQ$:
\begin{align} 
-\frc{28}{11},-\frc{65}{22}; -\frc{59}{23},-\frc{135}{46}; -\frc{65}{107},-\frc{1047}{214};'
\frc{33}{124},-\frc{715}{124}; -\frc{251}{169},-\frc{1357}{338}. 
\end{align}
This list has only one positive value, representing a Shabat polynomial defined over $\QQ$. 
Rarity of positive values for $p/r$ is observed in \cite{LandoZvonkin}. Here are the next few:
\[  
\frc{2008145}{1653242},\frc{1317156026567}{649800344821},
 \frc{487834953714776556005}{104793834699948131134},
 \frc{1999297558019926898176821908516}{208841813498019535906845150263}.
\]
Following (\ref{eq:ecp5}), the positivity region in Figure \ref{fig:positivepq}\refpart{i} is cut out by the lines $v=0$ 
and $15u+11v=645$. We have 3 separate small regions on $\EE_5$ between these lines, 
the one with $u>0$ especially tiny. The rational points are distributed ergodically on $\EE_5$,
with the density proportional to invariant holomorphic differential $du/v$. 
The whole measure is the real period over the finite oval or the infinite piece.
It can be computed avoiding numeric issues near $v=0$ or $u=\infty$ 
by integrating $du/v$ between points that differ by a 3-torsion point,
and then multiplying by 3. For example,
\begin{align}
\varrho_5 = 
 3 \int_{-45}^{-5} \frac{du}{v} = 3 \int_{51}^{315} \frac{du}{v}  \approx 0.732116211.
\end{align}
The integrals over the 3 small regions could be stably computed by shifting the intgeration range by a 3-torsion point.
The two integrals on the finite oval evaluate to $\approx 0.0564864103+0.0524120276$. 
After division by $\varrho_5$ we can compute the percentage and conclude that positive values of $p/q$
asymptotically occur about $5.72$ less frequently than the negative ones on the finite oval. 
The points (\ref{eq:mw1}) with $|n|<200$ give 46 and 43 positive values in the two regions.
This is within the rounding error from the asymptotic prediction.
The tiny integral on the infinite branch is $\approx 0.00407438266$.
This gives the odds ratio $\approx179$ for positive values. 
The points (\ref{eq:mw1}) with $|n|<200$ give 3 positive values of $p/q$, 
with $n\in\{54,162,-188\}$. Their denominators have 736, 6640 or 8944 decimal digits, respectively.
The same search looked for complete factorizations of the cubic polynomial in $z$, but only predictable degenerations were found.
By (\ref{eq:z23}) and (\ref{eq:b52uvs}), we need $\sqrt{102-2u}\in \QQ$ then.
\end{example}

\begin{figure}
\begin{center}
\begin{picture}(400,165)
\put(-3,2){\includegraphics[width=164pt]{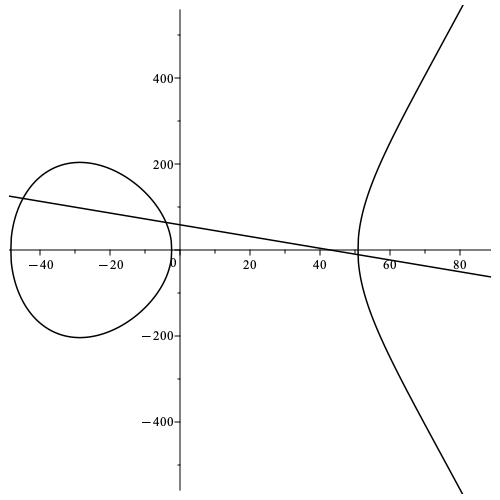}}
\put(236,2){\includegraphics[width=164pt]{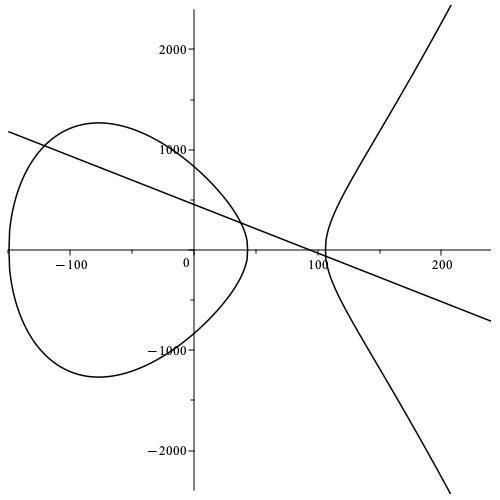}}
\put(0,-5){\refpart{i}} \put(237,-5){\refpart{ii}} 
\end{picture}
\end{center}
\caption{The elliptic curves (\ref{eq:ecm5}) and (\ref{eq:ecm6}), 
with the numerator line in (\ref{eq:ecp5}) or (\ref{eq:ecp6}) 
that together with the horizontal line $v=0$ determines the sign of $p/r$.}
\label{fig:positivepq}
\end{figure}

\begin{example} \rm \label{ex:ec6}
Similarly, for  $m=6$ we can investigate the cubic relation 
\begin{align} \label{eq:hec6}
\hpg21{-3,-\frac{7}2}{-\frac{p}{r}-6\,}{\,z}=0
\end{align}
between $p/r$ and $z=4\beta/\alpha^2$
by applying Theorem \ref{th:hypergell} with $b=-7/2$, $c=4+p/r$.
Additionally, we apply the simple transformation
\begin{equation}. \label{eq:b52uvs6}
(u,v)\mapsto \left( \frac{u-147}{4}, \, \frac{v}{8} \right)
\end{equation}
and derive the elliptic curve
\begin{align} \label{eq:ecm6}
\EE_6: \quad v^2=u^3-17787u+692566,
\end{align}
with an isomorphism
\begin{align}  \label{eq:ecp6}
\frac{p}{r}=\frac{5901-63u-13v}{4v},\qquad  z= \frac{791-13u+v}{8v}.
\end{align}
The data base of elliptic curves \cite[13230.dp1]{LMFDB}
tells that the Mordell-Weil group is $\ZZ\times(\ZZ/3\ZZ)$ here as well. 
The generators of $\QQ$-rational points are $(u,v)=(35,336)$ and a torsion point $(147,1120)$.
The elliptic involution acts as
\[
\left( \frac{p}{r},\,z \right) \mapsto \left( -\frac{p}{r}-\frac{13}2,\, 1-z \right).
\]
Here are some values of $p/r$ that give Belyi maps defined over $\QQ$:
\begin{align} 
-\frc{13}{4}\!,\,1,-\frc{15}{2},-\frc{13}{5},-\frc{39}{10},-\frc{30}{11},-\frc{83}{22},-\frc{52}{37},-\frc{377}{74},
-\frc{20}{29},-\frc{337}{58},
\frc{41}{16},-\frc{145}{16}. 
\end{align}
Positive values appear more frequently than in the $m=2$ case.
Besides the discardable $p/r=1$ representing a quadratic factor of $1-16x^8$ in $\QQ[x]$, 
we see $\frc{41}{16}$ as well. Here are the next few positive values: 
\[
\frc{83031}{5198},\frc{572982199}{35574034}, \frc{66779978696204}{16641846371989},
\frc{77057440650930450189}{148795415472621031586},
\frc{559984659717697802475460}{36618976557904027122191}.
\]
Following (\ref{eq:ecp6}), the positivity region in Figure \ref{fig:positivepq}\refpart{i} is cut out by the lines $v=0$ 
and $63u+13v=5901$. Again, we have 3 separate small regions on $\EE_6$ between these lines. 
The whole real period can be computed to be 
\begin{align}
 \varrho_6 = 
 3 \int_{-133}^{\,35} \frac{du}{v} = 3 \int_{107}^{707} \frac{du}{v}  \approx 0.541858251.
\end{align}
The two integrals on the finite oval evaluate to $\approx 0.0507070923+0.0430448636$. 
After division by $\varrho_6$ we can compute the odds ratio for positive $p/r$ to be $\approx 4.78$
for a rational point with $u<50$.  The small integral on the infinite branch is $\approx 0.00766222865$,
giving the odds ratio $\approx 69.7$. For a complete factorization of the cubic  polynomial 
we need $\sqrt{30(107-u)}\in\QQ$, but only a few predictable degenerations were found.
\end{example}

\begin{example} \rm \label{ex:ec7}
The case $m=7$ gives a genus 3 polynomial relation 
\begin{align} \label{eq:hec7}
\hpg21{-4,-\frac{7}2}{-\frac{p}{r}-7\,}{\,z}=0
\end{align}
between $p/r$ and $z=4\beta/\alpha^2$. 
The Faltings theorem \cite{Bombieri90}
implies that the genus 3 curve has only finitely many $\QQ$-rational points.
Section \ref{sec:hpg4fb} implies that it could be projected to an elliptic curve (\ref{eq:ec4sy})
with $b=-7/2$, which is
\begin{align}  \label{eq:ecm7}
\EE_7:\quad v^2=u\,\big(u^2-119u+\frc{14175}4\big).
\end{align}
The data base of elliptic curves \cite[94080.el2]{LMFDB} tells 
that the Mordell-Weil group of this curve has rank 2, and is isomorphic to $\ZZ^2\times(\ZZ/2\ZZ)$.
Beside the 2-torsion point $(u,v)=(0,0)$ and the generator $(\frc{105}2,-\frc{105}2)$ specialized from (\ref{eq:h4eg}),
another free generator is $(60,15)$.  To obtain rational points on the genus 3 curve,
we need rational points $(u,v)$ on $\EE_7$ giving a full square in (\ref{eq:h4tbs}).
That means $(v-14u)^2-11340u\,$ must be a full square.
An extensive search through $2\cdot61\cdot51>6000$ points on $\EE_7$ gives the following 12 values of 
$p/r=c-4$, paired by (\ref{eq:h4fbc}):
\begin{align} 
1,-\frc{17}2; 10,-\frc{35}{2}; -\frc{7}{3},-\frc{31}{6}; -\frc{14}{5},-\frc{47}{10}; -\frc{79}{11},-\frc{7}{22}; -\frc{34}{11},-\frc{97}{22}.
\end{align}
The value $1$ should be discarded along with the degenerate $p/r\in\{-5,-6,-7$, $-\frac12,-\frac32,-\frac52\}$.
The positive value 10 gives the Shabat polynomial {\small
\begin{equation}
(1+2x+4x^2)^{10\,} (1-20x+180x^2-880x^3+1760x^4+6336x^5-59840x^6+183040x^7). \,
\end{equation} }
\end{example}

\begin{example} \rm \label{ex:ec8}
Similarly, the case $m=8$ gives a genus 3 polynomial relation 
\begin{align} \label{eq:hec8}
\hpg21{-4,-\frac{9}2}{-\frac{p}{r}-8\,}{\,z}=0
\end{align}
between $p/r$ and $z=4\beta/\alpha^2$. 
Section \ref{sec:hpg4fb} implies that this curve could be projected to an elliptic curve (\ref{eq:ec4sy})
with $b=-9/2$, which is
\begin{align}  \label{eq:ecm8}
\EE_8:\quad v^2=u\,\big(u^2-243u+\frc{59535}4\big).
\end{align}
The data base of elliptic curves \cite[40320.bf2]{LMFDB} tells 
that the Mordell-Weil group is isomorphic to $\ZZ^2\times(\ZZ/2\ZZ)$.
Beside the generator $(\frc{189}2,-\frc{567}2)$ specialized from (\ref{eq:h4eg}),
another free generator is $(\frac{945}4,\frac{14175}8)$.  
We need rational points $(u,v)$ on $\EE_8$ giving a full square in (\ref{eq:h4tbs}).
That means $(v-18u)^2-34020u\,$ must be a full square.
An extensive search through $2\cdot61\cdot41>5000$ points on $\EE_8$ gives these 13 pertinent 
values of  $p/r=c-5$ from (\ref{eq:h4fbc}):
\begin{align} 
\hspace{-3pt} -\frc{17}2; -14,\frc{11}{2}; -\frc{8}{3},-\frc{35}{6}; 
-\frc{16}{5},-\frc{53}{10}; -\frc{17}{5},-\frc{51}{10}; -\frc{17}{7},-\frc{85}{14}; -\frc{371}{151},-\frc{1825}{302}.
\end{align}
The positive value $p/r=\frac{11}2$ gives the Shabat polynomial {\small
\begin{equation}
(1+2x+5x^2)^{11}\,(1-11x+44x^2-715x^4+2717x^5-572x^6-29172x^7+97240x^8)^2. 
\end{equation} }
\end{example}

\subsection{Cases with fewer Belyi maps}

\begin{example} \rm 
Let us examine the cases when the hypergeometric polynomials in (\ref{eq:q2pwhp}) or (\ref{eq:pqj2}) 
are linear in $z=4\beta/\alpha^2$.
For odd $m>2$ we obtain a single Belyi map 
when, respectively,
\begin{equation}
\ell=-\frac{p}{r}\in\left\{\frac{m+3}2,m-2\right\}.
\end{equation}
Indeed, setting $\ell=(m+3)/2$ in (\ref{eq:q2pwhp}) gives $z=-m-1$.
Then $H_2(x)=1+2x-(m+1)x^2$ up to $x$-scaling.
According to (\ref{eq:g2hmz}), the terms to $x^{2k}$ in the power series
\begin{align}
& \left(1+2x-(m+1)x^2\right)^{\!\frac{m+3}2} = \textstyle 1+(m+3)x-\frac{(m+1)(m+2)(m+3)}3\,x^3 \nonumber \\
& \qquad\textstyle -\frac{(m+1)(m+2)(m+3)}{12}(m-3)x^4+\frac{(m+1)(m+2)(m+3)(m-1)(m+5)}{20}\,x^5 \quad \nonumber  \\
& \qquad\textstyle +\frac{(m+1)(m+2)(m+3)^2(m-1)}{45}(m-5)x^6  \\ 
& \qquad\textstyle -\frac{(m+1)(m+2)(m+3)(m-1)(m-3)(m^2+19m+42)}{252}\,x^7 \nonumber \\ 
& \qquad\textstyle -\frac{(m+1)(m+2)(m+3)(m-1)(m-3)(m+5)(11m+26)}{3360}(m-7)\,x^8 +\ldots. \nonumber
\end{align}
must be divisible by $(m-2k+1)$. 
Setting $\ell=m-2$ in (\ref{eq:pqj2}) gives $z=m+2$. 
Here as well, the terms to $x^{2k}$ in the power series
\begin{align}
& \left(1+2x+(m+2)x\right)^{\frac{m-2}2} = \textstyle 1+(m-2)x+(m-2)(m-1)x^2+\frac{2m(m-2)(m-4)}3x^3 \nonumber \\
& \qquad\textstyle +\frac{(m-2)(m-4)(5m+2)}{12}(m-3)\,x^4+\frac{(m-2)(m-4)(m-6)(13m^2-9m-10)}{60}\,x^5 \nonumber  \\
& \qquad\textstyle +\frac{(m-2)(m-4)(m-6)(19m^2+5m-6)}{180}(m-5)\,x^6+\ldots. 
\end{align}
must be hypnotically divisible by $(m-2k+1)$. For even $m>2$ we obtain isolated Belyi maps 
when 
\begin{equation}
\ell=-\frac{p}{r}\in\left\{\frac{m+4}2,m-3\right\}.
\end{equation}
We obtain then $z=-m/3$ or $z=m/3+1$ from (\ref{eq:q2pwhp}) or (\ref{eq:pqj2}), respectively.
The terms to $x^{2k+1}$ in the following power series must be divisible by $(m-2k)$:
\begin{align}
& \left(1+2x-\!\frc{m}{3}x^2\right)^{\!\frac{m+4}2} = \textstyle 1+(m+4)x+\frac{(m+3)(m+4)}3\,x^2 
 -\frac{m(m+2)(m+3)(m+4)}{36}\,x^4 \nonumber \\
& \qquad\textstyle -\frac{m(m+2)(m+3)(m+4)}{180}(m-4)\,x^5 
 +\frac{m(m+2)(m+3)(m+4)(m^2+15m-36)}{1620}\,x^6  \nonumber  \\ 
& \qquad\textstyle +\frac{m(m-2)(m+2)(m+3)(m+4)(m+6)}{2835}(m-6)\,x^7+\ldots,\\
& \left(1+2x+\!\frc{m+3}{3}x^2\right)^{\!\frac{m-3}2} = \textstyle 1+(m-3)x+\frac{2(m-3)^2}3 x^2
+\frac{(m-3)(m-5)}{3}(m-2) x^3 \nonumber \\
& \qquad\textstyle  +\frac{(m-3)(m-5)(5m^2-33m+36)}{36}\,x^4 
 +\frac{(m-3)(m-5)(m-7)(m-1)}{20}(m-4)\,x^5   \nonumber  \\ 
& \qquad\textstyle  +\frac{(m-3)^2(m-5)(m-7)(13m^2-105m+72)}{810}\,x^6  \nonumber \\
& \qquad\textstyle +\frac{(m-3)(m-5)(m-7)(m-9)(53m^2-153m+72)}{11340}(m-6)\,x^7+\ldots.
\end{align}
\end{example}

\begin{example} \rm \label{ex:pell}
Let us look at the cases when the hypergeometric polynomials in (\ref{eq:q2pwhp}) or (\ref{eq:pqj2}) 
are quadratic in $z=4\beta/\alpha^2$.
For odd $m$ we then have 
\begin{equation}
\ell=-\frac{p}{r}\in\left\{\frac{m+5}2,m-4\right\}.
\end{equation}
Setting $\ell=(m+5)/2$ in (\ref{eq:q2pwhp}) gives this equation for $z$:
\begin{equation} \label{eq:beta2a}
3z^2+6(m+1)z+m^2-1=0.
\end{equation}
It has roots in $\QQ$ when the discriminant $24(m+1)(m+2)$ is a square. 
An equivalent condition is existence of integer solutions of the Pell equation
\begin{equation} \label{eq:pell1}
(2m+3)^2-6d^2=1,
\end{equation}
as observed in \cite[Ch.~10]{APZ}.
These solutions correspond to the units in the field $\QQ(\sqrt{6})$, which are generated by $5+2\sqrt{6}$.  
We should express them as 
\begin{equation}  \label{eq:pell1a}
\big(5+2\sqrt{6}\big)^n = (2m+3)\pm\sqrt{6}\,(z+m+1),
\end{equation}
leading to two integer values of $z$ for suitable $m$.
For $n\in\{2,3,4\}$ we get the suitable values $m\in\{23,241,2399\}$.
The case $m=23$ gives two Belyi maps with $p/r=-14$ defined $\QQ$,
with $H_2(x)=1+x-x^2$ or $H_2(x)=1+x-11x^2$.

Setting $\ell=m-4$ in  (\ref{eq:pqj2}) gives this equation for $z$:
\begin{equation}
3z^2-6(m+2)z+(m+2)(m+4)=0. 
\end{equation}
It has the same discriminant as (\ref{eq:beta2a}), leading to the same Pell equation (\ref{eq:pell1}).
The case $m=23$ gives two Belyi maps with $p/r=-19/2$ defined $\QQ$,
with $H_2=1+2x+5x^2$ or $H_2=1+2x+45x^2$.

For even $m$ we have 
\begin{equation}
\ell=-\frac{p}{r}\in\left\{\frac{m+6}2,\frac{m-5}2\right\}.
\end{equation}
Setting $\ell=(m+6)/2$ in (\ref{eq:q2pwhp}) gives this equation for $z$:
\begin{equation} \label{eq:beta2b}
15z^2-10mz+m(m-2)=0. 
\end{equation}
It has roots in $\QQ$ when the discriminant $40m(m+3)$ is a square. 
An equivalent condition is existence of integer solutions of the equation
\begin{equation} \label{eq:pell2b}
(2m+3)^2-10d^2=9.
\end{equation}
If $m$ is divisible by 3 (and thus by 6), we have a reduction to Pell's equation
\begin{equation}
\big( {\textstyle \frac23m+1 } \big )^2 - 10 \hat{d}^{\,2}=1.
\end{equation}
The solutions correspond to units in the field $\QQ(\sqrt{10})$, which are generated by $3+\sqrt{10}$.  
We should express the solutions as 
\begin{equation}  \label{eq:pell2a}
\big(3+\sqrt{10}\big)^n = \textstyle \big(\frac23m+1\big)\pm\sqrt{10}\,\big(z+\frac13m\big),
\end{equation}
leading to two integer values of $z$ for suitable $m$.
But the norm of $3+\sqrt{10}$ in $\QQ(\sqrt{10})$ is $-1$ rather than $1$; thus $n$ should be even.
Further, $m$ is prescribed to be even (while $n=2$ gives $m=27$, for example).
For that we need $n$ to be divisible by 4. The smallest possibility $n=4$ gives $m=1080$.

If $m$ is not divisible by $3$, equation (\ref{eq:pell2b}) is solved by considering the numbers of the norm 3
in $\QQ(\sqrt{10})$:
\begin{align}  \label{eq:pell2c}
(1+\sqrt{10})\big(3+\sqrt{10}\big)^n, \qquad
(1-\sqrt{10})\big(3+\sqrt{10}\big)^n. 
\end{align}
The two options correspond to the fact that $\QQ(\sqrt{10})$ is not a unique factorization domain;
its class number equals 2. In the first case, we need $n=3$ mod 4 for even $m$;  the smallest $n=3$ gives $m=242$. 
In the second case we need $n=1$ mod 4;  the smallest interesting $n=5$ gives $m=4802$. 

Setting $\ell=m-5$ in  (\ref{eq:pqj2}) gives this equation for $z$:
\begin{equation}
15z^2-10(m+3)z+(m+3)(m+5)=0.
\end{equation}
It has the same discriminant as (\ref{eq:beta2b}), leading to the same equation (\ref{eq:pell2b}).
\end{example}

\begin{example} \rm
Let us consider hypergeometric polynomials (\ref{eq:q2pwhp}) of degree 3 or 4.
Note that the lower parameter $1+k-\ell$ is supposed to be a positive integer for $k=m+1$.
With odd $m$ and $\ell=(m+7)/2$, the polynomial in $z=4\beta/\alpha^2$ is
\begin{equation}
\hpg21{-3,\,-\frac{5}2\,}{\ell-5}{\,z}=0.  
\end{equation}
Comparing with (\ref{eq:hec5}), we conclude that we need integer values $p/r=-\ell$
of Example \ref{ex:ec5} satisfying $p/r<-5$.  That excludes the discarded $p/r\in\{0,-1,-2,-3$, $-4,-5\}$ again.
The number of points on $\EE_5$ that give integer values of $p/r$ (or any other regular function)
is finite by Siegel's theorem \cite[IX.3]{Silverman}.
In this case, there appear to be no relevant integer values.

Similarly, consideration of integer $\ell\in \{\frac{m+8}2,\frac{m+9}2,\frac{m+10}2\}$ in (\ref{eq:q2pwhp}) 
leads  to the hypergeometric polynomials
 \begin{equation}
 \hpg21{-3,\,-\frac{7}2\,}{\ell-6}{z}\!, \qquad
\hpg21{-4,\,-\frac{7}2\,}{\ell-7}{z}\!, \qquad
\hpg21{-4,\,-\frac{9}2\,}{\ell-8}{z}\!, 
\end{equation}
comparable to (\ref{eq:hec6}), (\ref{eq:hec7}), (\ref{eq:hec8}), respectively.
We need sufficiently negative integer values of $p/r=-\ell$ from Examples \ref{ex:ec6}\,--\,\ref{ex:ec8}.
Only the value \mbox{$p/r=-14$} of Example \ref{ex:ec8} suits us. 
The pivotal hypergeometric evaluation is 
\begin{equation}
\hpg21{-4,\,-\frac{9}2\,}{6}{-4}=0.
\end{equation}
This gives $\ell=14$, $m=18$ and $H_2(x)=1+x-x^2$. 
The expansion of $H_2(x)^{14}$ misses the term with $x^{19}$ accordingly.
The polynomial $G_{18}(x)$ is defined by the lower degree terms: {\small
\begin{align}
& 1+14x+77x^2+182x^3-910x^5-1365x^6+1430x^7+5005x^8-10010x^{10}-3640x^{11} \nonumber \\
& +14105x^{12}+6930x^{13}-15625x^{14}-6930x^{15}+14105x^{16}+3640x^{17}-10010x^{18}.
\end{align}}The terms with $x^4,x^9$ are missed as well.
The broken symmetry around the term with $x^{14}$ is notable.
The Belyi map is $H_2(x)^{14}/G_{18}(x)$. 
\end{example}

\begin{example} \rm
Hypergeometric polynomials (\ref{eq:pqj2}) of degree 3 or 4 are considered similarly.
Consideration of odd $\ell\in\{m-6,m-7,m-8,m-9\}$ in (\ref{eq:pqj2})
leads  to the hypergeometric polynomials
\begin{equation*}
\hpg21{\!-3,-\frac{5}2}{\!-\frac{\ell}{2}-6}{z}\!, \qquad
\hpg21{\!-3,-\frac{7}2}{\!-\frac{\ell}{2}-7}{z}\!, \qquad
\hpg21{\!-4,-\frac{7}2}{\!-\frac{\ell}{2}-8}{z}\!, \qquad
\hpg21{\!-4,-\frac{9}2}{\!-\frac{\ell}{2}-9}{z}\!,
\end{equation*}
comparable to (\ref{eq:hec5}), (\ref{eq:hec6}), (\ref{eq:hec7}), (\ref{eq:hec8}), respectively.
The relation with $p/r$ of Examples \ref{ex:ec5}\,--\,\ref{ex:ec8} is always $p/r=1+\ell/2$.
We look for positive half-integer values $p/r$. 
Again, only Example \ref{ex:ec8} provides: $p/r=11/2$.
This gives $\ell=9$, $m=18$, and also $H_2(x)=1+2x+5x^2$. 
The expansion of $H_2(x)^{9/2}$ misses the term with $x^{19}$ indeed. We have
{\small
\begin{align}
G_{18}(x)= &\, 1+9x+54x^2+210x^3+630x^4+1386x^5+2394x^6+2394x^7+2655x^8 \\
& +1195x^9+252x^{10}-252x^{11}+168x^{12}-180x^{14}+228x^{15}-18x^{16}-378x^{17}+560x^{18}. \nonumber
\end{align}}For some reason, the coefficients $2394$ and $\pm252$ repeat twice consequently. 
The Belyi map is $G_{18}(x)^2/H_2(x)^{9}$. 
\end{example}

\small 
\bibliographystyle{alpha}

\end{document}